\documentclass{article}

\usepackage{amssymb, amsmath, amscd, amstext, amsfonts, color, amsthm}

\usepackage{xypic}
\input{xy}
\xyoption{all} 

\title{A commutative version of the group ring}

\author{W.H. MANNAN\footnote{School of Mathematics, University of Southampton, Southampton, SO\textup{17 1}BJ.  \newline \noindent e-mail \textup{: \texttt{wajid@mannan.info}}}}

\newtheorem{theorem}{${}$\hspace{-4.5mm}Theorem}[section]
\newtheorem{corollary}[theorem]{${}$\hspace{-4.5mm} Corollary}

\newtheorem{definition}[theorem]{${}$\hspace{-4.5mm} Definition}
\newtheorem{lemma}[theorem]{${}$\hspace{-4.5mm} Lemma}
\newtheorem{conjecture}[theorem]{${}$\hspace{-4.5mm}Conjecture}
\newtheorem{pot}[theorem]{${}$\hspace{-4.5mm}Potential counterexample}

\newcommand {\F} {\mathbb F}
\newcommand {\Z} {\mathbb Z}

\newcommand {\R} {\mathbb R}

\newcommand {\kG} {\kappa[G]}
\newcommand {\KG} {{\kappa[G]}^\#}

\newcommand {\kF} {\kappa[F_I]}
\newcommand {\KF} {{\kappa[F_I]}^\#}

\newcommand {\kH} {\kappa[H]}
\newcommand {\KH} {{\kappa[H]}^\#}

\newcommand{\vecl}[1] {\stackrel{\to}{#1}}

\newcommand{\barl}[1] {\overline{#1}}

\begin{document}

\maketitle

\begin{abstract}
We construct a commutative version of the group ring and show that it allows one to translate questions about the normal generation of groups into questions about the generation of ideals in commutative rings.  We demonstrate this with an alternative proof of a result about the normal generation of the free product of two cyclic groups.
\end{abstract}

 \section{Introduction}\label{intro}

A number of outstanding problems arising in low dimensional topology and combinatorial group theory share a common core difficulty; it is in general hard to show that a normal subgroup $H$ of a  group $G$ cannot be normally generated by a set of elements, $X \subset H$, with some specified property.  The most obvious obstruction is that if  $H$ is normally generated by $X$, then the image of $X$  would generate the image of $H$ in the  abelianisation, $G/[G,G]$.  However a new obstruction is required to address problems where abelianisation is too blunt.

Fix a field $\kappa$.  The normal subgroup structure of a group $G$ is reflected in the ideal structure of its group ring $\kG$, in a fairly simple way:  if $h_1,\cdots,h_n$ normally generate a normal subgroup $H \lhd G$, then the elements $1-h_1,\cdots,1-h_n$ generate the kernel of the induced map $\kG \to \kappa[G/H]$ as a two-sided ideal.  However even for a finitely presented group $G$, the group ring $\kG$ need not be either Noetherian or commutative, making its ideals harder to work with.

The approach of this paper is to instead work with a commutative version of the group ring (taken over a field $\kappa$ of characteristic 0).  This provides a functor from the category of groups to the category of commutative rings.  Questions about the normal generation of subgroups are thus turned into questions about the generation of ideals.  We demonstrate how non-trivial group theoretic results may thus be obtained from elementary commutative algebra.

We note that our construction has similarities to various existing algebraic structures.  For example the identities in Lemma                  \ref{lem13} and Corollary \ref{cor3} suggest a relationship with Clifford algebras, whilst Lemma \ref{lem10} suggests a relationship with the trace polynomials of Horowitz \cite{Horo}.  One can think of our construction as `less free' than these.  Another commutative analogue of group rings is given in \cite{Wood}.  

We first outline some problems which require a new such obstruction.  Given a finite set of generators for a finitely presented group $G$,
let $F$ be the free group on the generators and let $G \cong F/R$
where $R \lhd F$. {\bf The Relation Gap problem} \cite{Brid} asks: Must
the minimal number of elements needed to normally generate $R$ in $F$ equal the minimal number of elements needed to normally generate a subgroup 
of $R$ in $F$, which surjects onto the abelianisation $R/[R,R]$.

{\bf Wall's D2 problem} is a major unanswered question in low
dimensional topology.   It asks if cohomological dimension and
geometric dimension agree for homotopy types of finite cell
complexes.  C.T.C. Wall proved that they do agree, except possibly
for cell complexes of geometric dimension 3 and cohomological
dimension 2 \cite[Theorem E]{Wall}. However this remaining case has resisted attack for the last forty five years.  It has long been known that a counterexample resolving the Relation Gap problem would under certain hypotheses also solve Wall's D2 problem \cite{Harl}.  More recently these hypotheses have been reduced \cite{Mann}.  	In \cite{Mann1} Wall's D2 problem is expressed in terms of normal generation.

{\bf The Kervaire conjecture} posits that the free product of the integers with a non-trivial group cannot be normally generated by a single element.  

{\bf The Wiegold problem} \cite[Question 5.52]{Kour1} asks if a
finitely generated perfect group must be normally generated by one
element.  For a finite perfect group it is elementary to show that it is the normal closure of a single element.  It is easy to find finitely generated perfect groups where this is apparently not the case, yet there do not currently exist obstructions which would allow us to prove this.

A  conjecture due to H. Short \cite[Conjecture 2]{Howi}  posits: The number of elements needed to normally
generate a free product of $n$ (non-trivial) cyclic groups is at least $n/2$.  For
$n=3$ this was known for forty years as the Scott-Wiegold conjecture \cite[Problem 5.53]{Kour}, and was eventually proven \cite{Howi} by Jim Howie.  This problem is too subtle to approach by abelianisation as the free product of 3 cyclic groups may abelianise to a cyclic group (which is generated by a single element).  In this case it was resolved using topological methods.

The $n=4$ case of the above conjecture is trivially implied by the $n=3$ case.  If it could be proven for $n=5$ then the Wiegold problem would be solved: 

\begin{pot}{\rm
Let $G = C_p  \ast C_q \ast
C_r  \ast C_s  \ast C_t / \langle \pi \rangle$, where $\pi$ is the product of the generators of the
cyclic factors, and $p,q,r,s,t$ are pairwise coprime.  $G$ is perfect (by the Chinese remainder Theorem) but if it were the normal closure of a single element $x$, then the free product of 5 cyclic groups would be normally generated by just 2 elements: $\pi, x$.  Note that $G$ here is generated by $4$ generators, as the triviality of $\pi$ in $G$ implies that we may express one generator in terms of the other $4$.
\label{pot1}}
\end{pot}

There are also many potential counterexamples which could solve Wall's D2 problem and / or the Relation Gap problem \cite{Beyl, Brid, John}.  We mention a candidate due to recently published work by  Gruenberg and Linnell \cite{Grun}:

\begin{pot}\label{pot2}
{\rm
 Fix coprime integers $p,q$ and let $G=(C_p \times
\Z) \ast (C_q \times \Z)$.  Let $F$ be the free group on the natural
4 generators of $G$ and let $G \cong F/R$, where $R \lhd F$.   From \cite[Proposition
1.2]{Grun} we know that 3 elements of $R$ may be chosen so that their
normal closure in $F$ surjects onto the quotient $R/[R,R]$. However
it appears that $R$ cannot be normally generated by fewer than 4
elements.}  
\end{pot}

If this could be proven then the Relation Gap problem
would be solved. Further, if it was shown that every finite
presentation of $G$ required at least as many relators as generators, then Wall's D2 problem would also be solved.  Without loss of generality, the generators in such a presentation would be the natural 4 generators, together with a finite set of trivial generators (representing $e\in G$).

In the context of our methods, the Relation Gap and Wiegold problems could be approached by computing the commutative rings for  the groups from the examples above, and showing that the relevant ideal cannot be generated in an undesirable way.

Although we do not attempt to attack these problems in this paper, we are able to demonstrate the method by proving a similar result (originally due to Boyer \cite{Boye}); the free product of two cyclic groups cannot be normally generated by a proper power.  As in the examples above, this problem cannot be approached by abelianising the group, as the abelianisation would be cyclic if the orders of the cyclic factors are coprime.

In \S\ref{const} we define the commutative version of a group ring.  In \S\ref{iden}, \S\ref{desc} we discuss its properties, culminating in a complete description.  In \S\ref{cyc} we compute this for the free product of two cyclic groups.  We show that in the light of our methods, Boyer's result reduces to an absolutely elementary statement in commutative algebra.   Finally we show that the Scott-Wiegold conjecture may also be reduced to a statement about a commutative ring.  This time however the statement is not immediately obvious.

Having applied our commutative version of group rings to problems concerning $2$-generated and $3$-generated groups, it is our hope that these methods may be applied to the $4$-generated groups in Potential counterexamples \ref{pot1} and \ref{pot2}.  In this way we hope that our methods will eventually be able to resolve the major questions which we have mentioned here.

 \section{Construction of the commutative group ring}\label{const}

Fix a field $\kappa$ of characteristic 0.  All ring homomorphisms of $\kappa$-- algebras will be understood to fix $\kappa$.  The inverse operation on a group  $G$ extends  $\kappa$--linearly to an involution $*\colon \kG \to \kG$.  The invariants of $*$, $\kG^*$, are then a subset of $\kG$. The set of commutators $\left[\kG,\,\kG^*\right]$ generate a 2-sided ideal:  $$\langle\left[\kG,\, \kG^*\right]\rangle \lhd \kG.$$

\begin{definition}
Let $A_G= \kG / \langle[\kG,\, \kG^*]\rangle.$   
\end{definition}

\bigskip
Thus $A_G$ is the group ring of $G$ with the
added relations that the invariants of $*$ are central.  Given $x\in \kG$ we denote the corresponding element in $A_G$ by $[x]$.  

\begin{lemma}{The action of $*$ is well defined on $A_G$.}\label{lem1}\end{lemma}

\begin{proof}{For $z\in [\kG,\,\kG^*]$, we have $z= ab-ba$ for some $b$ satisfying $b^*=b$.  Thus $$z^* = (ab-ba)^*=b^* a^* - a^*b^* = b a^* - a^* b \in \langle[\kG,\,\kG^*]\rangle.$$  If $[x]=[y]$ then $x-y \in \langle  [\kG,\,\kG^*] \rangle$, so $(x-y)^* \in \langle  [\kG,\,\kG^*] \rangle$ and $[x^*] =[y^*]$.}
\hfill ${}$ \end{proof}

Let $\KG = {A_G}^*$, the invariants of this action.  Let $\Lambda_G$ be the anti-invariants of the action: $\Lambda_G=\{[x]\in A_G\vert\,\,[x]^*=-[x]\}$.  For $x \in \kG$ let: $$\bar{x} = \frac{x+x^*}{2} \in \kG^*,\qquad \vec{x} = \frac{x-x^*}{2}\in\kG,$$
so $[\bar{x}]\in \KG,\,\, [\vec{x}]\in\Lambda_G$.

\begin{lemma}{We have a decomposition of $\kappa$--linear vector spaces{\rm  :} $$A_G = \KG \oplus \Lambda_G.$$}\label{lem2}\end{lemma}

\begin{proof}{If $[x]\in \KG \cap \Lambda_G$ then $[x]=[x]^*=-[x]$ so $[x]=0$.  Given any $[y]\in A_G$ we have $[y]=[\bar{y}]+[\vec{y}]$}. \hfill${}$\end{proof}

\begin{lemma}
If $[x] \in \KG$ then $[x]=[\bar{x}]$ and if $[x] \in \Lambda_G$ then $[x]=[\vec{x}]$.
\label{2.=xbar}
\end{lemma}

\begin{proof}
If $[x] \in \KG$ then $$[x]=[\bar{x}+\vec{x}]=[\bar{x}]+\frac{[x -x^*]}{2}=[\bar{x}]+\frac{[x] -[x]^*}{2}=[\bar{x}].$$ 
If $[x] \in \Lambda_G$ then $$[x]=[\bar{x}+\vec{x}]=\frac{[x +x^*]}{2}+[\vec{x}]=\frac{[x] +[x]^*}{2}+[\vec{x}]=[\vec{x}].$$
\end{proof}

As $\bar{x} \in \kG^*$, we have $[\bar{x}]$ central in $A_G$.  Thus $\KG$ is central in $A_G$.

\begin{lemma}{$\KG$ is a subring of $A_G$ and $\Lambda_G$ is a (right) module over $\KG$.}\label{lem3}\end{lemma}

\begin{proof}{ Clearly $1    =     [e]  \in  \KG$ and $\KG$ is closed under addition and subtraction.  If $[x], [y] \in \KG$                                 
 then $[x][y]    \in   \KG$ as $$([x][y])^*   =   [(xy)^*]    =     [y^*x^*]    =     [y]^*[x]^*                 =  [y][x]    =  [x][y].$$   Thus $\KG$ is also closed under multiplication, and therefore a subring of $A_G$.

To see that $\Lambda_G$ is a module over $\KG$, we check that $\Lambda_G$ is closed under multiplication by elements of $\KG$.  To that end let $[x]\in \KG, [y]\in\Lambda_G$.  Then:  
$$([y][x])^*=[(yx)^*]=[x^*y^*]=[x]^*[y]^* =-[x][y]=-[y][x].$$ so $[y][x] \in \Lambda_G$.} 
\end{proof}

\noindent As $\KG$ is central in $A_G$, we know that $\KG$ is a commutative ring.  

\begin{definition}{We define the commutative version of the group ring for the group $G$, over a field $\kappa$, to be $\KG$.} \label{comver}\end{definition}

We note the following identities in $\KG$:
\begin{lemma}{Let $[x],[y],[z] \in A_G$.  Then{\rm :}                                  

\noindent i) $\,\,[\overline{xy}]=[\overline{yx}]$. 
                                 
\noindent ii) $\,2[\bar{y}]\,[\overline{xz}]=[\overline{xyz}]+[\overline{xy^*z}]$, {\rm so in particular} $[\bar{x}][\bar{y}]=\frac12([\overline{xy}]+[\overline{xy^*}])$.
                                 
\noindent iii) $[\overline{x^*}]=[\bar{x}]$.
}
\label{lem10}\end{lemma}

\begin{proof}{ i) We have: $$\overline{xy}-\overline{yx} = (\bar{x}y-y\bar{x})+(\bar{y}x^*-x^*\bar{y}).$$
As $[\bar{x}],[\bar{y}] \in \KG$ we conclude $[\overline{xy}]-[\overline{yx}]=[0]$.

\bigskip                                 
\noindent  ii) We have $$\overline{xyz}+\overline{xy^*z}= x\bar{y}z + z^* \bar{y}x^*.$$
As $[\bar{y}]\in \KG$ we conclude $[\overline{xyz}]+[\overline{xy^*z}]=2[\bar{y}]\,[\overline{xz}]$. 

\bigskip                                
\noindent  iii) Trivial.}
\hfill${}$\end{proof}

\bigskip
Let $f\colon G \to H$ be a group homomorphism.  Then $f$ extends  $\kappa$--linearly  to a ring homomorphism $f\colon \kG \to \kH$.  We will now show that this induces a ring homomorphism: $$f^\#\colon \KG\to \KH.$$

\begin{lemma}{$f$ induces a well defined ring homomorphism $\hat{f}\colon A_G \to A_H$.}\label{lem4}\end{lemma}

\begin{proof}{From the definition of group homomorphism we know that $( f(x))^*=f(x^*)$ for all $x \in G$ and hence for all $x\in \kG$. Thus if $x\in \kG^*$, then $$( f(x))^*   = f( x^*)  = f(x).$$ Thus $f(x) \in \kH^*$ and if $z                 =  yx -    xy$ for some $y   \in    \kG$, then $$f(z)     =    f(y)f(x)-  f(x)f(y)    \in  [\kH,\,\kH^*].$$  Hence if $[x]=[y]$, then $[f(x)]=[f(y)]$ and we may define $\hat{f}([x])=[f(x)]$.
}
${}$\hfill${}$\end{proof}

Clearly $\hat{f}([x]^*)= (\hat{f}([x]))^*$, so if $[x]\in\KG$ then $$(\hat{f}([x]))^*=\hat{f}([x]^*)=\hat{f}([x])$$ and $\hat{f}([x])\in \KH$.  Thus we may define:

\begin{definition}{For a group homomorphism $f  \colon G \to  H$, we define the map of commutative rings $f^\#  \colon \KG                                                   \to   \KH$ to be the restriction of $\hat{f}$} to $\KG$.\label{func}\end{definition}

By construction  $(fl)^\#=f^\# l^\#$, where $l$ is a group homomorphism $H \to M$, so we have a functor 
${\rm \bf Group}\to{\rm \bf Commutative\, Ring}$, sending a group $G$ to the commutative  ring $\KG$, and the group homomorphism, $f\colon G \to H$ to the ring homomorphism $f^\#\colon \KG \to \KH$.  This functor has the following `half-exactness' property:

\begin{lemma}{If $f\colon G \to H$ is a surjective group homomorphism, then the ring homomorphism $f^\#\colon \KG \to \KH$ is also surjective.}\label{lem5}\end{lemma}

We note in passing that the dual statement is not true.  For example let $D_8$ denote the group of symmetries of a square and let $V$ be the subgroup generated by reflections through lines parallel to the edges.   Then the inclusion $\iota\colon V \hookrightarrow D_8$ is an injective group homomorphism, but $\iota^\#$ is not injective.

\bigskip
Proof of Lemma \ref{lem5}:$\,\,\,$ {The ring homomorphism $f\colon \kG \to \kH$ is surjective so the induced map $\hat{f} \colon A_G \to A_H$ is also surjective.  Hence given $[x]\in \KH$ we have $y \in \kG$ with $\hat{f}([y])=[x]$.  Then $$f^\#([\bar{y}])=[f(\bar{y})]=\frac12([f(y)]+[f(y^*)])=\frac12(\hat{f}([y])+ (\hat{f}([y]))^*)=\frac12([x]+[x]^*) =[x]$$ as required. 
}                  \hfill $\Box$
\,\,

\bigskip
For the remainder of this section  $f\colon G \to H$ is a surjective group homomorphism, with kernel $K \lhd G$.  We wish to find generators for the ideal  ker$(f^\#) \lhd\KG$.

\begin{lemma}{As vector spaces over $\kappa$:

\noindent i) $\,\,{\rm  ker}(f\colon \kG\to \kH)$ is generated by the set  $S=\{gk-g \vert\, g \in G, k\in K\}$.
                                 
\noindent ii) $\,{\rm  ker}(\hat{f}\colon A_G \to A_H)$ is generated by the set $S'=\{[gk]-[g] \vert\, g \in G, k\in K\}$.
                                 
\noindent iii) ${\rm  ker}(f^\#\colon \KG  \to  \KH)$ is generated by the set  $S''=\{[\overline{gk}]-[\bar{g}] \vert\, g \in G, k\in K\}$.                                 }

\label{lem6}\end{lemma}

\begin{proof} { i) In $\kG  / \langle S \rangle \kappa$ there is precisely one way of representing each element of $\kH$.

\bigskip                                                                  
\noindent ii) We know that $ S' \kappa$ is a 2-sided ideal in the ring $A_G$  as   
$$([gk]-[g])[h]=[gh(h^{-1}kh)]-[gh],\qquad [h]([gk]-[g])=[hgk]-[hg]$$ for $g,h \in G,\, k \in K$.  Thus $A_G/ \langle S' \rangle \kappa$ is a ring.
From (i) we know that $A_G/ \langle S' \rangle \kappa$ may be identified with a quotient ring of $\kH$.  

To see that this quotient ring is $A_H$, we check that the additional relation in $A_H$: $$f(x)=( f(x))^* \implies [f(x)]\,\, {\rm central},$$ is  satisfied in $A_G/ \langle S' \rangle \kappa$.  That is we need: $$f(x)= (f(x))^* \implies  [x]\in A_G/ \langle S' \rangle \kappa {\rm \, is\, central}.$$

Note that given $x\in \kG$ with $f(x)=(f(x))^*$, we have $f(\bar{x})=f(x)$. Thus from (i) we have $[x] \sim [\bar{x}]$ in $A_G/ \langle S' \rangle \kappa$ and $[\bar{x}]$ is already central in $A_G$.
     
\bigskip                                                             
\noindent iii) We have $[gk]-[g]=[\overline{gk}]-[\bar{g}] +[\vecl{gk}]-[\vec{g}]$, so by (ii), ker$(\hat{f})$ is generated over $\kappa$  by:                                  
$$ S''  \cup \{[\vecl{gk}]-[\vec{g}] \vert\,\, g \in G,\, k\in K\}.$$  Thus if $[x] \in {\rm ker}(f^\#)$, then $[x]=[y]+[z]$, where $[y] \in  \langle S'' \rangle \kappa$ and $z\in \Lambda_G$.  We have $[z]=[x]-[y]\in \KG$, so by Lemma \ref{lem2} we know $[z]=0$.
}\hfill${}$\end{proof}

For $x \in A_G$, $[\bar{x}] =\frac12([x]+ [x]^*)$, which depends only on $[x]$, so we have a well defined $\kappa$--linear map $m_G\colon A_G \to \KG$ given by $ [x] \mapsto  [\bar{x}]$.

\begin{lemma}{The following maps are linear over $\KG${\rm:}
                                 
\noindent i) The map $A_G \to A_G$ given by $[x] \mapsto [xg]$, for some fixed $g \in G$.
                                 
\noindent ii) The map $m_G\colon A_G \to \KG$.}

\label{lem7}\end{lemma}

\begin{proof}{i) Multiplication by a ring element is always linear over a central subring.                                 

\noindent
ii) Let $[x]\in A_G$ and $[y] \in \KG$.  We must check that $m_G([x][y])=m_G([x])[y]$: 
                                 $$m_G([x][y])=\frac12([x][y]+([x][y])^*)= \frac12([x]+[x]^*)[y]= m_G([x])[y].$$
}
\end{proof}

Let $B \subset A_G$ be a generating set for $A_G$ as a module over $\KG$.

\begin{lemma}{Given $l\in G$,  the ideal  in $\KG$ generated by $\{[\overline{bl}]                 -                 [\bar{b}] \vert\, [b]                  \in                  B\}$ contains $[\overline{xl}]-[\bar{x}]$ for any $x \in A_G$.}

\label{lem8}\end{lemma}

\begin{proof} By Lemma \ref{lem7} we have a linear map $t_l\colon A_G \to \KG$ given by: $$[y] \mapsto [\overline{yl}]-[\bar{y}].$$  We may  express $[x]$  as a sum $[x]=[b_1][y_1]+\cdots +[b_r][y_r]$ with  the  $[b_i]\in B, [y_i] \in \KG$.  Then applying  $t_l$ to both sides gives  $$[\overline{xl}]-[\bar{x}]=([\overline{b_1l}]-[\bar{b_1}])[y_1]+\cdots+([\overline{b_rl}]-[\bar{b_r}])[y_r].$$

\end{proof}

\begin{definition}{
Given a set $L \subset G$, we define $L^\#\lhd \KG$ to be the ideal generated by $\{[\overline{bl}]-[\bar{b}] \vert\, [b] \in B, l\in L\}$. }\label{gen}\end{definition}

By Lemma                  \ref{lem8}, $L^\#$ is independent of the choice of generating set $B$.  In fact we could take a different generating set of $A_G$, $B_l$ for each $l\in L$ and would still have:

\begin{lemma}{The ideal $L^\#\lhd \KG$ is generated by $\{[\overline{bl}]-[\bar{b}] \vert\,  l\in L,[b] \in B_l\}$.}
\label{lem9}\end{lemma}

\begin{theorem}{If the set $L\subset K$ normally generates $K$, then  ${\rm ker}(f^\#)=L^\#$.}\label{thm:norm}\end{theorem}

\begin{proof}{As $L\subset K$ it is clear that $L^\#\subset{\rm ker}(f^\#)$.  Let $K'$ denote the set of $k \in K$ such that for all $g \in G$, $[\overline{gk}]-[\bar{g}] \in L^\#$.  By Lemma                  \ref{lem8} we have that $L \subset K'$.  Clearly $[\bar{g}]-[\bar{g}]=0 \in L^\#$ for all $g \in G$, so $e \in K'$.

Suppose $k_1, k_2 \in K'$.  Then for all $g \in G$ we have $$ [\overline{gk_1k_2}]-[\bar{g}]= ([\overline{(gk_1)k_2}]                 -                 [\overline{gk_1}])+([\overline{gk_1}]                 -                 [\bar{g}])  \in L^\#.$$   Thus $k_1k_2 \in K'$.

Also if $k\in K'$, then for all $g \in G$ we have  $$[\overline{gk^{-1}}] - [\bar{g}]                  =                  -( [\overline{(gk^{-1})k}] - [\overline{gk^{-1}}]) \in L^\#.$$  Thus $k^{-1} \in K'$.
                                 
                  Let $h\in G$ and suppose $k\in K'$.  Then for all $g \in G$, using Lemma \ref{lem10}i  we have $$[\overline{g(hkh^{-1}})]-[\bar{g}]=[\overline{(h^{-1}gh)k}]-[\overline{h^{-1}gh}] \in L^\#.$$   Thus $hkh^{-1} \in K'$.
                                 
                   So $K'$ contains $L$ and the identity, and is closed under multiplication, inverse and conjugation.  As $L$ normally generates $K$, we have $K'=K$.  So $L^\#$ contains $[\overline{gk}]-[\bar{g}]$ for all $g\in G$ and $k \in K$.  Thus by Lemma                  \ref{lem6}iii we obtain $L^\#={\rm ker}(f^\#)$.
}\hfill${}$\end{proof}

Theorem \ref{thm:norm} fails if we drop the condition that $f$ is a surjective group homomorphism.  For example, if we let $f\colon V \to D_8$ be the inclusion mentioned before, then we may take $L=\phi$ the empty set, to normally generate ${\rm ker}(f)$ (as $f$ is injective).  However ker$(f^\#)$ is non-zero in this case.

Theorem \ref{thm:norm} is at the heart of our approach.  If $K$ is normally generated by $N \subset K$ and we wish to show that some other set $L\subset K$ does not normally generate $K$, then by Theorem              \ref{thm:norm}, $N^\#= {\rm ker}(f^\#)$ and it is enough to show that $L^\# \neq N^\#$.

For this purpose, it will sometimes be more convenient to use a refinement of $L^\#$.  Let $B$ now be a generating set for the module $\Lambda_G$ over $\KG$.  Note that by Lemma \ref{lem2}, the set $B \sqcup \{1\}$ generates the module $A_G$ over $\KG$. 

\begin{definition}
Given a set $L \subset G$ we define $L^{\#\#} \lhd \KG$ to be the ideal generated  by $\{[\barl{bl}]-[\barl{bl^{-1}}]\vert \,\,[b]\in B, \, l \in L\}$.\label{defhashhash}
\end{definition}

\begin{lemma}{Given $l\in G$,  the ideal  in $\KG$ generated by $\{[\overline{bl}]                 -                 [\barl{bl^{-1}}] \vert\, [b]                  \in                  B\}$ contains $[\overline{xl}]-[\barl{x^{-1}}]$ for any $[x] \in A_G$.}

\label{2.genx}\end{lemma}

\begin{proof}
  We proceed  as in Lemma  \ref{lem8} by expressing $[x]$ as a linear combination of the elements of $B \sqcup \{[1]\}$ and hence express $[\overline{xl}]-[\barl{xl^{-1}}]$ as a linear combination of the  $\{[\overline{bl}]                 -                 [\barl{bl^{-1}}] \vert\, [b]                  \in                  B\}$ and $[\barl{1l}]-[\barl{1l^{-1}}]$ which equals $0$ by Lemma \ref{lem10}iii.
\end{proof}

 Thus $L^{\#\#}$ is independent of the choice of $B$, a generating set for $\Lambda_G$.  Indeed, as before one may chose a different generating set $B_l$ for each $l \in L$ and still have $L^{\#\#}$ generated by $\{[\overline{bl}]-[\barl{bl^{-1}}] \vert\, l\in L, [b] \in B_l\}$

\begin{theorem}
If $L$ normally generates a normal subgroup $K \lhd G$, then $L^{\#\#}=K^{\#\#}$.
\label{hashhash}
\end{theorem}

\begin{proof}
This proof follows the structure of the proof of Theorem \ref{thm:norm}.   Let $K'$ denote the set of $k \in K$ such that for all $g \in G$, $[\overline{gk}]-[\barl{gk^{-1}}] \in L^{\#\#}$.  By Lemma                  \ref{2.genx} we have that $L \subset K'$.  As before it is clear that $e \in K'$ and that $K'$ is closed under taking inverses.  In order to show that $K'=K$ and hence prove the theorem, it remains to show that $K'$ is closed under group multiplication and conjugation.

Given $g,h \in G$ and $k\in k'$ from Lemma \ref{lem10}i we again have: $$
[\overline{g(hkh^{-1})}]-[\barl{g(hk^{-1}h^{-1})}]=[\overline{(h^{-1}gh)k}]-[\barl{(h^{-1}gh)k^{-1}}]\in L^{\#\#}.
$$
Thus $K'$ is closed under conjugation.

Given $g \in G$ and $j,k \in K'$,  we may expand $[\barl{k}]\, ([\barl{gj}]-[\barl{gj^{-1}}])$ using  Lemma \ref{lem10}ii:$$
[\barl{k}]\, ([\barl{gj}]-[\barl{gj^{-1}}])=\frac{[\barl{gjk}]+[\barl{gjk^{-1}}]}2 -\frac{[\barl{gkj^{-1}}]+[\barl{gk^{-1}j^{-1}}]}2.
$$
Set $z=\frac{gj+j^{-1}g}2$.  Then using Lemma \ref{lem10}i we obtain:

$$
[\barl{g(jk)}]-[\barl{g(jk)^{-1}}]=[\barl{k}]\, ([\barl{gj}]-[\barl{gj^{-1}}])+[\barl{zk}]-[\barl{zk^{-1}}]\in L^{\#\#}.
$$
Thus $K'$ is closed under multiplication.

\end{proof}

Thus given a normal subgroup $K\lhd G$, normally genrated by a set of elements $N \subset K$, we may use Theorem \ref{hashhash} to show that $K$ is not normally generated by a set $L\subset K$.  All one need do is show that $L^{\#\#}\neq N^{\#\#}$.  In Section \ref{cyc} we apply this method to provide an alternative proof of a result due to Boyer \cite{Howi}.  In Section \ref{Scott} we apply this method to the Scott-Wiegold conjecture \cite{Boye}.  Before that, we must be able to compute the commutative ring $\KG$ and the relevant ideals $L^{\#\#}$.  We build up the theory to do that in Section \ref{iden} and Section \ref{desc}.

Finally in this section, we will comment on the relationship between $L^\#$ and $L^{\#\#}$.  Given a set $L\subset G$, one might naively generalize from the case of the usual group ring, and consider  $L^\bullet \lhd \KG$ to be the relevant ideal in relation to the normal closure of $L$ (rather than $L^\#$ or $L^{\#\#}$), where $L^\bullet$ is the ideal generated by $\{1-[\bar{l}]\vert\,\,l\in L\}$.  These three ideals are related as follows:

\begin{lemma}
We have $L^\#= L^{\#\#} \stackrel\cdot+ L^\bullet$.
\end{lemma}

\begin{proof}
Let $g\in G$ and $l \in L$.  Applying Lemma \ref{lem10}ii we get:$$
[\barl{gl}]-[\bar{g}]=\frac{[\barl{gl}]-[\barl{gl^{-1}}]}2+[\bar{g}]([\bar{l}]-1)\in  L^{\#\#} \stackrel\cdot+ L^\bullet.$$
\end{proof}

 \section{Identities in $A_G$}\label{iden}

In order to implement our approach, we need to have a description of the commutative ring $\KG$ and we need a generating set $B$ for $A_G$ as a module over $\KG$.  Then given $K \lhd G$, a set $N \subset K$, which normally generates $K$, and another set $L \subset K$, we would have generators for the ideals $N^\#, L^\#$ (resp. $N^{\#\#}, L^{\#\#}$).   Working out whether or not  $L^\# = N^\#$  (resp. $N^{\#\#}=L^{\#\#}$)is then purely a problem in commutative ring theory.  If then $L^\# \neq N^\#$ (or $N^{\#\#}\neq L^{\#\#}$), we may conclude that $L$ does not normally generate $K$.

In \S\ref{desc} we will show in general how to compute $\KG$ and $B$, from a presentation of $G$.  However to do this we first need to build up a collection of identities which hold in $A_G$.  We will now drop the square parentheses which we have used to distinguish elements of $\kG$ and elements of $A_G$.  Thus from now on, given $x\in \kG$ it will be understood that $x$ may also denote $[x]\in A_G$ depending on context.

By Lemma \ref{2.=xbar}, we know that if $x \in \KG$ then $x=\bar{x}$, and if $y\in \Lambda_G$ then $y=\vec{y}$ (noting dropped parentheses).  Conversely for all $x\in A_G$ we have $\bar{x}\in\KG, \, \vec{x}\in\Lambda_G$.  So we may denote a general element of $\KG$ by $\bar{x}$ and a general element of $\Lambda_G$ by $\vec{x}$.

\begin{lemma}{Given $\vec{x},\vec{y} \in \Lambda_G$ we have, $\vec{x}\vec{y}+\vec{y}\vec{x} \in \KG$, $\vec{x}\vec{y}-\vec{y}\vec{x} \in \Lambda_G$.}
\label{lem11}\end{lemma}

\begin{proof}{We need only note that $(\vec{x}\vec{y})^*={\vec{y}}^{\,*}{\vec{x}}^{\,*}=\vec{y}\vec{x}$
and similarly $(\vec{y}\vec{x})^*=\vec{x}\vec{y}$.}
\hfill${}$\end{proof}

\begin{definition}{We define a dot product $\Lambda_G \times \Lambda_G \to \KG$ and bracket $\Lambda_G \times \Lambda_G \to \Lambda_G${\rm:} $$\vec{x} \cdot \vec{y}=-\frac12(\vec{x}\vec{y}+\vec{y}\vec{x}),\qquad[\vec{x}, \vec{y}]=\frac12(\vec{x}\vec{y}-\vec{y}\vec{x}).$$}\label{dot}\end{definition}

\begin{lemma}{ For $\vec{x},\vec{y},\vec{z}\in \Lambda_G$ we have{\rm :} 
                                 
\noindent i) $\,\,\,\,\,\vec{x}\cdot\vec{y}= -\frac12(\barl{xy}-\barl{x y^*})=\bar{x}\bar{y}-\barl{xy}$,
                                 
\noindent ii) $\,\,[\vec{x},\vec{y}]=\frac12(xy-yx)$ {\rm (so in particular $xy-yx \in \Lambda_G$)},
                                 
\noindent iii) $[\vec{x},\vec{y}]\cdot\vec{z}=-\frac12(\barl{xyz}-\barl{zyx})$.
}
\label{lem12}\end{lemma}

\begin{proof}{i)  We have 
\begin{eqnarray*}\vec{x}\cdot\vec{y}&=& -\frac18\left((x-x^*)(y-y^*)+(y-y^*)(x-x^*)\right)\\&=&-\frac14(\barl{xy}+\barl{yx}-\barl{x y^*}-\barl{y^*x}\,)= -\frac12({\barl{xy}-\barl{x y^*}} ).
\end{eqnarray*}
using Lemma \ref{lem10}i.  Then Lemma \ref{lem10}ii gives: $$-\frac12({\barl{xy}-\barl{x y^*}} ) =  -\frac12(2\barl{xy}-2\bar{x}\bar{y}) =\bar{x}\bar{y}-\barl{xy}.$$
                                                                  
\noindent ii) As $\bar{x},\,\bar{y}$ are central in $A_G$, we have:
$$
\vec{x}\vec{y}-\vec{y}\vec{x}=(x-\bar{x})(y-\bar{y})-(y-\bar{y})(x-\bar{x})=xy-yx.
$$
 
%\begin{eqnarray*}
%\vec{x}\vec{y}-xy = \frac14((x-x^*)(y-y^*)-4xy)=\frac{1}{2}(\bar{x}y^*-2x\bar{y}-\bar{x}y) \\ =\frac{1}{2}(y^*\bar{x}-2\bar{y}x- y\bar{x})=  \frac14((y-y^*)(x-x^*)-4yx)=\vec{y}\vec{x}-yx,
%\end{eqnarray*} as $\bar{x},\bar{y}$ are central in $A_G$.  Thus $\vec{x}\vec{y}-\vec{y}\vec{x}=xy-yx$.
                                                                  
\noindent iii) From (i) and (ii) we know that: $$[\vec{x},\vec{y}]\cdot\vec{z}=-\frac14(\barl{(xy-yx)z}-\barl{(xy-yx)z^*}). $$ Also from (ii) we know that $(xy-yx)^*=yx-xy$, so by Lemma                  \ref{lem10}iii: $$[\vec{x},\vec{y}]\cdot\vec{z}=-\frac14(\barl{(xy-yx)z}-\barl{z(yx-xy)})=-\frac12 (\barl{xyz}-\barl{zyx}).$$
}
\end{proof}

Recall Defintion \ref{defhashhash}.  Let $B$ be a generating set for the module $\Lambda_G$ over the ring $\KG$.  From Lemma  \ref{lem12}i it immediately follows that:

\begin{corollary}
Given a set $L \subset G$, we have the ideal $L^{\#\#} \lhd \KG$ generated by $\{\vec{b}\cdot\vec{l} \vert\,\, \vec{b}\in B,\, l\in L\}$. \label{dotishashhash}
\end{corollary}

Recall from Lemma                  \ref{lem3} that $\Lambda_G$ is a module over $\KG$.

\begin{lemma}{i) The dot product and the bracket are bilinear maps over $\KG$.
                                 
\noindent ii) The dot product is symmetric and the bracket is skew-symmetric.
                                 
\noindent iii) The bracket obeys the Jacobi identity: $[[\vec{x},\vec{y}],\vec{z}]+ [[\vec{y},\vec{z}],\vec{x}]+[[\vec{z},\vec{x}],\vec{y}]=0$. 
                                 
\noindent iv) The scalar triple product $(\vec{x},\vec{y},\vec{z}) \mapsto[\vec{x},\vec{y}]\cdot\vec{z}$, is an alternating trilinear map. 
                                 
\noindent v) The triple bracket may be expanded as follows{\rm :} $$ [[\vec{x},\vec{y}],\vec{z}]=(\vec{x}\cdot\vec{z})\vec{y}-(\vec{y}\cdot\vec{z})\vec{x}.$$
}
\label{lem13}\end{lemma}

\begin{proof} i) This follows from the centrality of $\KG$ in $A_G$.

\noindent ii) This is immediate from the definitions.

\noindent iii) This follows from the form of the definition and may be verified by calculation.

\noindent iv) Consider the form of the scalar triple product given by Lemma                  \ref{lem12}iii.  If any pair of $\vec{x},\vec{y},\vec{z}$ are equal then applying Lemma                  \ref{lem10}i if necessary, we get  $[\vec{x},\vec{y}]\cdot\vec{z}=0$. 

\bigskip                                                                   
\noindent v) We have $$
[[\vec{x},\vec{y}],\vec{z}]=\frac14\left( \vec{x}\vec{y}\vec{z}-\vec{y}\vec{x}\vec{z}-\vec{z}\vec{x}\vec{y}+\vec{z}\vec{y}\vec{x}   \right).
$$
On the other hand: $$(\vec{x}\cdot\vec{z})\vec{y}-(\vec{y}\cdot\vec{z})\vec{x}=-\frac12 ((\vec{x}\vec{z}+\vec{z}\vec{x})\vec{y}-\vec{x}(\vec{y}\vec{z}+\vec{z}\vec{y})) =\frac12 (\vec{x}\vec{y}\vec{z}-
\vec{z}\vec{x}\vec{y}).$$                                    
Thus taking the difference gives: 
$$
[[\vec{x},\vec{y}],\vec{z}]-((\vec{x}\cdot\vec{z})\vec{y}-(\vec{y}\cdot\vec{z})\vec{x})
=\frac14\left(
-(\vec{x}\vec{y}+\vec{y}\vec{x})\vec{z}+\vec{z}(\vec{x}\vec{y}+\vec{y}\vec{x})
\right)=0,
$$
as by Lemma \ref{lem11} we have $\vec{x}\vec{y}+\vec{y}\vec{x}\in \KG$ central.

%So $(\vec{x}\cdot\vec{z})\vec{y}-(\vec{y}\cdot\vec{z})\vec{x}- [[\vec{x},\vec{y}],\vec{z}]=
%\frac12 (\vec{x}\vec{y}\vec{z}-\vec{z}\vec{x}\vec{y}) -\frac14(\vec{x}\vec{y}\vec{z}-\vec{y}\vec{x}\vec{z} -
%\vec{z}\vec{x}\vec{y} + \vec{z}\vec{y}\vec{x})\\=\frac14(\vec{x}\vec{y}\vec{z}+\vec{y}\vec{x}\vec{z} -
%\vec{z}\vec{x}\vec{y} - \vec{z}\vec{y}\vec{x}) = \frac12 (\vec{z}(\vec{x}\cdot\vec{y})-(\vec{x}\cdot\vec{y}) \vec{z})=0$,  by Lemma  \ref{lem11}.

\end{proof}

\begin{corollary}{We may deduce that given $\vec{x},\vec{y},\vec{z},\vec{w} \in \Lambda_G$ we have{\rm :} 
                                 
\noindent i) $\,\,\,\,[\vec{x},\vec{y}]\cdot\vec{z}= [\vec{y},\vec{z}]\cdot\vec{x}=[\vec{z},\vec{x}]\cdot\vec{y}$.
                                 
\noindent ii) $\,\,\,[\vec{x},\vec{y}] \cdot [\vec{z},\vec{w}] = 
(\vec{x}\cdot \vec{z})(\vec{y}\cdot \vec{w})-(\vec{x}\cdot \vec{w})(\vec{y}\cdot \vec{z})$.
                                 
\noindent iii) $[[\vec{x},\vec{y}] , [\vec{z},\vec{w}]]=([\vec{x},\vec{z}] \cdot\vec{w})\vec{y}- ([\vec{y},\vec{z}] \cdot\vec{w})\vec{x}$

$\,\qquad\qquad\qquad=-([\vec{x},\vec{y}] \cdot\vec{z})\vec{w}+ ([\vec{x},\vec{y}] \cdot\vec{w})\vec{z}$ .

\noindent iv) $([\vec{x},\vec{y}]\cdot\vec{z})\vec{w}=(\vec{x}\cdot\vec{w})[\vec{y},\vec{z}] -(\vec{y}\cdot\vec{w})[\vec{x},\vec{z}]
+(\vec{z}\cdot\vec{w})[\vec{x},\vec{y}]$.
                                 
\noindent v)  $\,\,\,[[\vec{x},\vec{y}] , [\vec{z},\vec{w}]]=(\vec{x}\cdot \vec{z})[\vec{y},\vec{w}]+ (\vec{y}\cdot \vec{w})[\vec{x},\vec{z}]
-(\vec{x}\cdot \vec{w})[\vec{y},\vec{z}]- (\vec{y}\cdot \vec{z})[\vec{x},\vec{w}]$.
}
\label{cor3}\end{corollary}

\begin{proof}{i) As 3-cycles  are even permutations, this follows from Lemma \ref{lem13}iv.

\bigskip                                 
\noindent ii) From (i) we have $[\vec{x},\vec{y}] \cdot [\vec{z},\vec{w}] =  [[\vec{z},\vec{w}],\vec{x}]\cdot \vec{y}$. Then Lemma                  \ref{lem13}v gives $$ [[\vec{z},\vec{w}],\vec{x}]\cdot \vec{y}= ((\vec{z}\cdot\vec{x})\vec{w}-(\vec{w}\cdot\vec{x})\vec{z})\cdot \vec{y}= (\vec{x}\cdot \vec{z})(\vec{y}\cdot \vec{w})-(\vec{x}\cdot \vec{w})(\vec{y}\cdot \vec{z}).$$
                                                                   
\noindent iii) We expand using Lemma                  \ref{lem13}v, taking $[\vec{z},\vec{w}]$ as an input to get the first identity, and $[\vec{x},\vec{y}]$ as an input to get the second.
 
\bigskip                                                                  
\noindent iv) We use Lemma                  \ref{lem13}v to expand the quadruple product $[[[\vec{x},\vec{y}],\vec{w}],\vec{z}]$ in two ways:
\begin{eqnarray*}
{} [[[\vec{x},\vec{y}],\vec{w}],\vec{z}]=  ([\vec{x},\vec{y}]\cdot\vec{z})\vec{w}-(\vec{z}\cdot\vec{w})[\vec{x},\vec{y}],\\
{} [[[\vec{x},\vec{y}],\vec{w}],\vec{z}]= (\vec{x}\cdot\vec{w})[\vec{y},\vec{z}] -(\vec{y}\cdot\vec{w})[\vec{x},\vec{z}].\\
\end{eqnarray*}
Equating the two yields the result.

\bigskip                                                                  
\noindent v) From (iii) we have $[[\vec{x},\vec{y}] , [\vec{z},\vec{w}]]=-([\vec{x},\vec{y}] \cdot\vec{z})\vec{w}+ ([\vec{x},\vec{y}] \cdot\vec{w})\vec{z}$ .  We may then use (iv) to substitute in for the expressions $([\vec{x},\vec{y}] \cdot\vec{z})\vec{w}$ and $ ([\vec{x},\vec{y}] \cdot\vec{w})\vec{z}$. 
}\hfill${}$\end{proof}

We can recover multiplication in $A_G$ from the module structure of $\Lambda_G$ over the ring $\KG$, together with the dot product and bracket:

\begin{lemma}{Given $x,y \in A_G$ we may express their product in the following terms{\rm :}  $$xy=\bar{x}\bar{y}-\vec{x}\cdot \vec{y}+\bar{x}\vec{y}+\bar{y}\vec{x}+[\vec{x},\vec{y}]$$ Thus in particular $\barl{xy}=\bar{x}\bar{y}-\vec{x}\cdot\vec{y}$ and $\vecl{xy}= \bar{x}\vec{y}+\bar{y}\vec{x}+[\vec{x},\vec{y}]$.}
\label{lem15}\end{lemma}

\begin{proof}  We have $$xy\,\,=\,\,(\bar{x}+\vec{x})(\bar{y}+\vec{y})\,\,=\,\, \bar{x}\bar{y}+\vec{x}\vec{y}+\bar{x}\vec{y}+\bar{y}\vec{x}\,\,=\,\,\bar{x}\bar{y}-\vec{x}\cdot \vec{y}+\bar{x}\vec{y}+\bar{y}\vec{x}+[\vec{x},\vec{y}].$$  To get the expressions for $\barl{xy}$ and $\vecl{xy}$, note that $xy=\barl{xy}+\vecl{xy}$, and  by Lemma \ref{lem2} this is the unique decomposition as a sum of an element of $\KG$ and an element of $\Lambda_G$.
\hfill${}$\end{proof}

\begin{lemma}{For $\vec{x},\vec{y},\vec{z},\vec{w} \in \Lambda_G$ we may expand the following products in $A_G${\rm :}
                                                                  
\noindent i) $\,\,\,\,\vec{x}\vec{y}=-\vec{x}\cdot \vec{y}+[\vec{x},\vec{y}]$,
                                 
\noindent ii) $[\vec{x},\vec{y}]\vec{z}=- [\vec{x},\vec{y}]\cdot \vec{z}+(\vec{x}\cdot\vec{z})\vec{y}-(\vec{y}\cdot\vec{z})\vec{x}$,                                 

\hspace{0.1mm}
$\vec{z}[\vec{x},\vec{y}]= - [\vec{x},\vec{y}]\cdot \vec{z}-(\vec{x}\cdot\vec{z})\vec{y}+(\vec{y}\cdot\vec{z})\vec{x},$

\bigskip
\noindent iii) $[\vec{x},\vec{y}][\vec{z},\vec{w}]=-(\vec{x}\cdot \vec{z})(\vec{y}\cdot \vec{w})+(\vec{x}\cdot \vec{w})(\vec{y}\cdot \vec{z})+ ([\vec{x},\vec{z}] \cdot\vec{w})\vec{y}- ([\vec{y},\vec{z}] \cdot\vec{w})\vec{x}$

${}\qquad \qquad\quad = -(\vec{x}\cdot \vec{z})(\vec{y}\cdot \vec{w})+(\vec{x}\cdot \vec{w})(\vec{y}\cdot \vec{z})
- ([\vec{x},\vec{y}] \cdot\vec{z})\vec{w}+ ([\vec{x},\vec{y}] \cdot\vec{w})\vec{z}.$ 
}
\label{lem16}\end{lemma}

\begin{proof}{i)  This follows immediately from Definition \ref{dot}.
                                 
\noindent ii) By Lemma \ref{lem15} the products may be written:  $$
[\vec{x},\vec{y}]\vec{z} =  -[\vec{x},\vec{y}]\cdot \vec{z}+[[\vec{x},\vec{y}],\vec{z}],
\qquad 
\vec{z}[\vec{x},\vec{y}]= - [\vec{x},\vec{y}]\cdot \vec{z}-[[\vec{x},\vec{y}],\vec{z}].$$  
The triple bracket may then be expanded by Lemma \ref{lem13}v to give the result.
                                                                  
\bigskip
\noindent iii) From Lemma \ref{lem15} we have $$[\vec{x},\vec{y}][\vec{z},\vec{w}]=-[\vec{x},\vec{y}] \cdot [\vec{z},\vec{w}] 
+ [[\vec{x},\vec{y}] , [\vec{z},\vec{w}]].$$  
We then expand $[\vec{x},\vec{y}] \cdot [\vec{z},\vec{w}]$  by Corollary  \ref{cor3}ii and $[[\vec{x},\vec{y}] , [\vec{z},\vec{w}]]$ by Corollary \ref{cor3}iii.
}\hfill${}$\end{proof}

 Expanding products in $A_G$  can yield identities in $\KG$:

\begin{lemma}{For $x,y,z \in A_G$ we have $\frac12(\barl{xyz}+\barl{zyx})=\barl{xy}\,\bar{z}+\barl{xz}\,\bar{y}+\barl{yz}\,\bar{x}-2\bar{x}\,\bar{y}\,\bar{z}$.
}\label{lem17}\end{lemma}

\begin{proof}{By Lemma \ref{lem15} we have $$\barl{xyz}\,\,+\vecl{xyz}\,\,=\,\,xyz\,\,=\,\,(\bar{x}\bar{y}-\vec{x}\cdot \vec{y}+\bar{x}\vec{y}+\bar{y}\vec{x}+[\vec{x},\vec{y}])(\bar{z}+\vec{z}).$$  
Equating components in $\KG$, we get
$$\barl{xyz}\,\,=\,\,\bar{x}\,\bar{y}\,\bar{z}-(\vec{x}\cdot \vec{y})\bar{z} -(\vec{y}\cdot \vec{z})\bar{x}-(\vec{x}\cdot \vec{z})\bar{y}-[\vec{x},\vec{y}])\cdot \vec{z}\,.$$                                                    
Only the last term  changes sign when $x$ and $z$ are swapped, so $$\frac12(\barl{xyz}+\barl{zyx})=\bar{x}\,\bar{y}\,\bar{z}-(\vec{x}\cdot \vec{y})\bar{z} -(\vec{y}\cdot \vec{z})\bar{x}-(\vec{x}\cdot \vec{z})\bar{y}.$$
Finally, we may use Lemma \ref{lem12}i to substitute in expressions for the dot products.} \hfill${}$\end{proof}

\begin{lemma}{Given $\vec{x}, \vec{y},\vec{z},\vec{u},\vec{v},\vec{w}\in \Lambda_G$, we have the following identities in $\KG${\rm  :}                                                                  

\noindent i) $$([\vec{y},\vec{z}]\cdot \vec{w})(\vec{x} \cdot \vec{u})
-([\vec{x},\vec{z}]\cdot \vec{w})(\vec{y} \cdot \vec{u})
+([\vec{x},\vec{y}]\cdot \vec{w})(\vec{z} \cdot \vec{u})
-([\vec{x},\vec{y}]\cdot \vec{z})(\vec{w} \cdot \vec{u})=0.$$

\noindent ii) $$([\vec{x},\vec{y}]\cdot\vec{z})([\vec{u},\vec{v}]\cdot\vec{w})= 
\left| 
\begin{array}{ccc} 
\vec{x}\cdot\vec{u}\,\,&\,\, \vec{x}\cdot\vec{v}\,\,&\,\,\vec{x}\cdot\vec{w}\\
\vec{y}\cdot\vec{u}\,\,&\,\, \vec{y}\cdot\vec{v}\,\,&\,\,\vec{y}\cdot\vec{w}\\
\vec{z}\cdot\vec{u}\,\,&\,\, \vec{z}\cdot\vec{v}\,\,&\,\,\vec{z}\cdot\vec{w}\\
\end{array} \right|\,\,{}_{.}$$
}
\label{lem19}\end{lemma}

\begin{proof}{i) From Corollary \ref{cor3}iii we know 
$$([\vec{y},\vec{z}]\cdot \vec{w})\vec{x} 
-([\vec{x},\vec{z}]\cdot \vec{w})\vec{y}
+([\vec{x},\vec{y}]\cdot \vec{w})\vec{z}
-([\vec{x},\vec{y}]\cdot \vec{z})\vec{w} =0.$$   We then take  the dot product of both sides with $\vec{u}$ to obtain the result.

\bigskip                                                                  
\noindent ii) By Corollary \ref{cor3}ii we have \begin{eqnarray}[[\vec{x},\vec{y}],\vec{w}]\cdot [[\vec{u},\vec{v}],\vec{z}]\,\,=\,\,([\vec{x},\vec{y}]\cdot[\vec{u},\vec{v}])(\vec{z}\cdot\vec{w})-([\vec{x},\vec{y}]\cdot\vec{z})([\vec{u},\vec{v}]\cdot\vec{w}).\label{det}\end{eqnarray}                                     From Lemma \ref{lem13}v we have $$[[\vec{x},\vec{y}],\vec{w}]\,\,=\,\,(\vec{x}  \cdot \vec{w})\vec{y}-(\vec{y}\cdot \vec{w})\vec{x}\qquad {\rm and}\qquad [[\vec{u},\vec{v}],\vec{z}]\,\,=\,\,(\vec{u}  \cdot \vec{z})\vec{v}-(\vec{v}\cdot \vec{z})\vec{u}.$$  
Also by Corollary \ref{cor3}ii we have $$([\vec{x},\vec{y}]\cdot[\vec{u},\vec{v}])\,\,=\,\,(\vec{x}\cdot\vec{u})(\vec{y}\cdot\vec{v})-(\vec{x}\cdot\vec{v})(\vec{y}\cdot\vec{u}) .$$  
Substituting these three expressions into (\ref{det}) and rearranging gives:
\begin{eqnarray*}
([\vec{x},\vec{y}]\cdot\vec{z})([\vec{u},\vec{v}]\cdot\vec{w})&=&
\biggl((\vec{x}\cdot\vec{u})(\vec{y}\cdot\vec{v})-(\vec{x}\cdot\vec{v})(\vec{y}\cdot\vec{u})\biggr) 
(\vec{z}\cdot\vec{w})\\&-&
\biggl((\vec{x}  \cdot \vec{w})\vec{y}-(\vec{y}\cdot \vec{w})\vec{x}\biggr) \cdot
\biggl((\vec{u}  \cdot \vec{z})\vec{v}-(\vec{v}\cdot \vec{z})\vec{u}\biggr)
\end{eqnarray*}
Multiplying out the dot product then gives the result.
}\hfill${}$\end{proof}

This lemma will be sufficient for describing the ring structure of $\KG$ in \S\ref{desc}.  We note one further identity, this  time between dot products only:  

\begin{corollary}{Let $\vec{x}, \vec{y},\vec{z},\vec{w},\vec{u},\vec{v},\vec{s},\vec{t} \in \Lambda_G$.  We have{\rm :}$$
\left| 
\begin{array}{cccc} 
\vec{x}\cdot\vec{u}\,\,\,&\,\,\, \vec{x}\cdot\vec{v}\,\,\,&\,\,\,\vec{x}\cdot\vec{s}\,\,\,&\,\,\,\vec{x}\cdot\vec{t}\\
\vec{y}\cdot\vec{u}\,\,\,&\,\,\, \vec{y}\cdot\vec{v}\,\,\,&\,\,\,\vec{y}\cdot\vec{s}\,\,\,&\,\,\,\vec{y}\cdot\vec{t}\\
\vec{z}\cdot\vec{u}\,\,\,&\,\,\, \vec{z}\cdot\vec{v}\,\,\,&\,\,\,\vec{z}\cdot\vec{s}\,\,\,&\,\,\,\vec{z}\cdot\vec{t}\\
\vec{w}\cdot\vec{u}\,\,\,&\,\,\, \vec{w}\cdot\vec{v}\,\,\,&\,\,\,\vec{w}\cdot\vec{s}\,\,\,&\,\,\,\vec{w}\cdot\vec{t}\\
\end{array} \right| =0.
$$}\label{cor4}\end{corollary}

\begin{proof}{We expand the determinant along the top row, using Lemma \ref{lem19}ii to express the minors:                                   $$
{\Big (}[\vec{y},\vec{z}]\cdot \vec{w}{\Big)}{\Big (} 
([\vec{v},\vec{s}]\cdot \vec{t})(\vec{u}\cdot \vec{x})
-([\vec{u},\vec{s}]\cdot \vec{t})(\vec{v}\cdot \vec{x})
+([\vec{u},\vec{v}]\cdot \vec{t})(\vec{s}\cdot \vec{x})
-([\vec{u},\vec{v}]\cdot \vec{s})(\vec{t}\cdot \vec{x})
{\Big )}$$ Lemma \ref{lem19}i then implies that the second factor here is 0.
}\hfill${}$\end{proof}

We finish this section with some identities relating to powers and commutators of group elements.

\begin{definition}{For $n      \in    \Z$ let $P_n$ be the polynomial over $\kappa$ determined by: $$P_0         =        0,\qquad P_1       =        1, \qquad 2xP_n(x)=P_{n-1}(x)+P_{n+1 }(x).$$}\label{poly}\end{definition}
 
We may break down powers of group elements using Lemma \ref{lem10}ii:

\begin{lemma}{Given $g,h,k \in G$ we have $\overline{gh^nk}=\overline{ghk}P_n(\bar{h})-\overline{gk}P_{n-1}(\bar{h})$.} \label{cor1}\end{lemma}

\begin{proof}{The cases $n=0,1$ follow from $P_{-1}=-1, P_0=0,P_1=1$.  For $r \in \Z$ we have 
\begin{eqnarray*}\barl{gh^{r-1}k}&=&2\bar{h}\barl{gh^rk}-\barl{gh^{r+1}k},\\ \barl{gh^{r+2}k}&=&2\bar{h}\barl{gh^{r+1}k}-\barl{gh^{r}k},\end{eqnarray*}  by Lemma                  \ref{lem10}ii. Thus if the statement of the lemma  holds for $n=r, r+1$, then  it also holds for $n=r-1,r+2$ and we are done by induction.}
\end{proof}

\begin{lemma}
Given $g,h \in G$ we have $\barl{gh^n}-\barl{gh^{-n}}=(\barl{gh}-\barl{gh^{-1}})P_n(\bar{h})$.
\end{lemma}

\begin{proof}
Apply Lemma \ref{cor1} with $k=e$ to get: $\overline{gh^n}-\overline{gh^{-n}}=$ $$
(\overline{gh}P_n(\bar{h})-\overline{g}P_{n-1}(\bar{h}))-(\overline{gh^{-1}}P_n(\bar{h})-\overline{g}P_{n-1}(\bar{h}))=(\barl{gh}-\barl{gh^{-1}})P_n(\bar{h}).
$$ 
\end{proof}

\begin{corollary}
Given $h \in G$ the ideal $\{h^n\}^{\#\#}\lhd \KG$ is contained in the ideal generated by $P_n(\bar{h})$. \label{power}
\end{corollary}

This will be important when we consider free products of cyclic groups in Sections \ref{cyc} and \ref{Scott}, as they may be expressed as free groups, quotiented out by the normal subgroup generated by a power of each generator. 

\bigskip
Let $a,b \in G$ and let $I \lhd \KG$ be the ideal generated by scalar triple products $\{[\vec{a},\vec{b}]\cdot\vec{c}\vert\,\,\vec{c} \in \Lambda_G\}$.  The following lemma shows that if $x,y \in G$ map to the same group element when the relation that $a$ commutes with $b$ is added, then $\bar{x}-\bar{y}\in I$.

\begin{lemma}
We have an equality of ideals: $\{a^{-1}b^{-1}ab\}^\#=I$.\label{lem:com}
\end{lemma}

\begin{proof}
Given $\vec{c} \in \Lambda_G$, Lemma \ref{lem12}iii gives:
$$
[\vec{a},\vec{b}]\cdot\vec{c}=-\frac12 (\barl{cab}-\barl{cba})=-\frac12 (\,\barl{cba(a^{-1}b^{-1}ab)}-\barl{cba}\,)\in \{a^{-1}b^{-1}ab\}^\#.
$$
Thus $I \subset \{a^{-1}b^{-1}ab\}^\#$.  Conversely, given $c\in A_G$, Lemma \ref{lem12}iii gives:
$$
-\frac12 (\,\barl{c(a^{-1}b^{-1}ab)}-\barl{c}\,)=-\frac12 (\,\barl{(ca^{-1}b^{-1})ab}-\barl{(ca^{-1}b^{-1})ba}\,)=[\vec{a},\vec{b}]\cdot\vecl{(ca^{-1}b^{-1})}\in I.
$$ Thus $\{a^{-1}b^{-1}ab\}^\#\subset I$.
\end{proof}

\begin{lemma} i) We have: $[\vec{a},\vec{b}]\cdot[\vec{a},\vec{b}]=(\vec{a}\cdot\vec{a})(\vec{b}\cdot\vec{b})-(\vec{a}\cdot\vec{b})^2$.

\hspace{16mm}ii) Commutators in $\KG$ break down as follows: $$\overline{aba^{-1}b^{-1}}= 2 \overline{ab}^2- 4 \bar{a} \bar{b}\overline{ab}+2\bar{a}^2+2\bar{b}^2-1.$$\label{lem:comdotcom}
\end{lemma}

\begin{proof}
i) This follows immediately from Corollary {\ref{cor3}}ii.

\bigskip
ii) Lemma \ref{lem15} yields an expression for $\vecl{ba}$.  Noting that $[\vec{a},\vec{b}]\cdot\vec{a}=[\vec{a},\vec{b}]\cdot\vec{b}=0$ by Lemma \ref{lem13}iv, this gives:  $[\vec{a},\vec{b}]\cdot[\vec{a},\vec{b}]=-[\vec{a},\vec{b}]\cdot\vecl{ba}=[\vec{a},\vec{b}]\cdot \vecl{(a^{-1}b^{-1})}$.

Thus Lemma \ref{lem12}iii gives $$[\vec{a},\vec{b}]\cdot[\vec{a},\vec{b}]= [\vec{a},\vec{b}]\cdot \vecl{(a^{-1}b^{-1})}=-\frac12(\,\barl{aba^{-1}b^{-1}}-1 \,).$$
Thus from (i): $$\barl{aba^{-1}b^{-1}}=1-2([\vec{a},\vec{b}]\cdot[\vec{a},\vec{b}])=1-2((\vec{a}\cdot\vec{a})(\vec{b}\cdot\vec{b})-(\vec{a}\cdot\vec{b})^2).$$

We complete the proof by substituting in the identities:$$
\vec{a}\cdot\vec{a}=-(\vec{a}\cdot\vecl{a^{-1}})=1-\bar{a}^2,\qquad \vec{b}\cdot\vec{b}=1-\bar{b}^2,\qquad 
\vec{a}\cdot\vec{b}=\bar{a}\bar{b}-\barl{ab},
$$ which all follow from Lemma \ref{lem12}i.
\end{proof}

 \section{A complete description of $\KG$}\label{desc}

By a complete description of $\KG$ we mean a set of elements which generate $\KG$ as a ring over $\kappa$, together with a defining set of relations which those generators satisfy.

Suppose that $K$ is a normal subgroup of $G$ with $G/K\cong H$, and suppose that $L\subset K$  normally generates $K$ in $G$.  By Theorem \ref{thm:norm} we know that if we have a complete description of $\KG$ and a generating set for $A_G$ over $\KG$, then we also have a complete description of $\KH=\KG/L^\#$.  We merely need to add the generating set for $L^\#$ (from Definition \ref{gen}) to the set of relations describing $\KG$.

Given an indexing set $I$, for each $i \in I$,  let $g_i\in F_I$ denote the corresponding letter in the free group on $I$.   Let $L=\{r_j \in F_I\vert\,\, j \in J\}$, for some indexing set $J$ and elements $r_J \in J$.  As before let $K\lhd G$ be the normal closure of $L$.

\begin{definition}{A presentation of $G$ is a pair of lists:  $\langle g_i, i\in I \vert\, r_j, j \in J \rangle$, together with a group isomorphism $F_I/K \cong G$ (with $L, K$ defined as above).}\label{pres}\end{definition}

We now fix a presentation $\langle g_i, i\in I \vert\, r_j, j \in J \rangle$ of a group $G$ with $K,L$ as before.  Elements of $F_I$ may  be regarded as elements of $G$, via the isomorphism $F_I/K\cong G$.   We will make clear which of these we mean from context.

The purpose of this section is to produce a complete description of $\KG$ from this data.   From the discussion above we only need a complete description of $\kappa[F_I]^\#$ and a generating set for $A_{F_I}$ as a module over $\KF$.  Then, using $\KG \cong \KF/L^\#$, we have a complete description of $\KG$.

If the indexing set $I$ is finite, then we will show that we only require a finite number of generators for $\KG$ as a ring over $\kappa$.  Further, if $J$ is also finite, then our description of $\KG$ will only have a finite number of relations.

We begin by defining an abstract ring $R_I$ which we will later  identify with $\kappa[F_I]^\#$.  Let $R_I$ denote the quotient of the polynomial ring $\kappa[\lambda_i,m_{ij}, w_{ijk} \vert\, i,j,k \in I]$ by the following relations for $i,j,k,l,s,t \in I$:

\bigskip

\noindent R1: $\qquad m_{ij}=m_{ji},\quad w_{ijk}=w_{kij},\quad w_{ijk}=-w_{kji}$

\noindent R2: $\qquad m_{ii}=1- \lambda_i^2$

\noindent R3: $\qquad w_{jkl}m_{is}-w_{ikl}m_{js}+w_{ijl}m_{ks}-w_{ijk}m_{ls}=0$

\bigskip
\noindent R4: $\qquad w_{ijk}w_{lst}=\left| 
\begin{array}{ccc} 
m_{il}& m_{is}&m_{it}\\
m_{jl}& m_{js}&m_{jt}\\
m_{kl}& m_{ks}&m_{kt}
\end{array} \right|$

\bigskip
\noindent We define a ring homomorphism $\phi\colon R_I \to \KG$ by:
\begin{eqnarray*}\lambda_i\,\,\,\,\,\,&\mapsto& \,\,\,\,\,\,\bar{g_i},\qquad\qquad\qquad\quad\,  i \in I,\\
 \phi\colon\qquad m_{ij} \,\,&\mapsto& \quad\vec{g_i} \cdot \vec{g_j}, \qquad\qquad\,\,\,\, i,j \in I,\\
w_{ijk}  &\mapsto&\,\,\,\,[\vec{g_i},\vec{g_j}]\cdot \vec{g_k}, \,\qquad i,j,k \in I.
\end{eqnarray*}

\begin{lemma}{The ring homomorphism $\phi$ is well defined.}\label{lem20}\end{lemma}

\begin{proof}{We must check that the relations R1, R2, R3, R4 are respected by $\phi$.  This is the case for R1 by Lemma                  \ref{lem13}ii, Lemma                  \ref{lem13}iv and Corollary \ref{cor3}i.   

It is also the case for R2 as by Lemma \ref{lem12}i we have $$\phi(m_{ii})= \vec{g_i} \cdot \vec{g_i}= - \vec{g_i} \cdot \vec{{g_i^{-1}}}= 1-\bar{g_i}^2=\phi(1-\lambda_i^2).$$   
Lemma                  \ref{lem19}i implies $\phi$ respects R3 and Lemma \ref{lem19}ii implies $\phi$ respects R4.
} \hfill${}$\end{proof}

\noindent Let $B$ denote the set: $\{1\} \cup \{\vec{g_i} \vert\, i\in I\} \cup \{[\vec{g_i},\vec{g_j}]\vert\,i,j \in I\} \subset A_G$.

\begin{lemma}{$\phi$ is surjective and $B$ is a generating set for $A_G$ as a module over $\KG$.}\label{lem21}\end{lemma}

\begin{proof}{We will show that any element of $A_G$ is a linear combination of elements of $B$ over $\phi(R_I)$.  Then given $x\in\KG$ we will have $x=1 \phi(\alpha)+y$, for some $\alpha \in R_I, y \in \Lambda_G$.  By Lemma \ref{lem2} $y=0$ and $x=\phi(\alpha)$ so we conclude $\phi$ is surjective.
                                                    
Let $\langle B \rangle\phi(R_I)$ denote the set of linear combinations of $B$ over $\phi(R_I)$.  Clearly  $\langle B \rangle\phi(R_I)$ is a vector space over $\kappa$.  For $i \in I$ we have $g_i,g_i^{-1} \in \langle B \rangle\phi(R_I)$ as $$g_i= 1 \phi(\lambda_i) + \vec{g_i},\qquad g_i^{-1}= 1 \phi(\lambda_i) - \vec{g_i}.$$
                                                    
Further, by Lemma \ref{lem16} we know $\langle B \rangle\phi(R_I)$ is closed under multiplication.
}\hfill${}$\end{proof}

Thus we have an explicit generating set $B$ for $A_G$ as a module over $\KG$.  Note that if $I$ is finite, then $B$ is also a finite set.  Further, given any finite set $S \subset G$, the ideal $S^\#$  is then generated by a finite set (given in Definition \ref{gen}).

Also, if $I$ is finite then $R_I$ is a finitely generated commutative ring over $\kappa$ and hence Noetherian.   By Lemma                  \ref{lem21} $\KG$ is a quotient of $R_I$, so we also have that $\KG$ is a Noetherian ring, finitely generated over $\kappa$.  This is a stark contrast with the usual group ring, which need not be Noetherian for a finitely generated group.

Let $C\subset\{g_ig_j \vert\, i,j \in I\}$ satisfy that for all $i,j \in I$, either $g_ig_j \in C$ or $g_jg_i \in C$.  
Let  $B'=\{1\} \cup \{{g_i} \vert\, i\in I\} \cup C \subset A_G$.  This can be a convenient alternative to $B$:

\begin{lemma}{$B'$ is also a generating set for $A_G$ as a module over $\KG$.
}
\label{lem22}\end{lemma}

\begin{proof}{For $i,j \in I$ we have $\vec{g_i}=-1\bar{g_i}+g_i$.   By Lemma \ref{lem15},
\begin{eqnarray*}
[\vec{g_i},\vec{g_j}]&=&  \,\,\, \, 1(\vec{g_i}\cdot\vec{g_j}-\bar{g_i}\bar{g_j} )+g_ig_j-\vec{g_i}\bar{g_j}-\vec{g_j}\bar{g_i}\\
&= &-1(\vec{g_i}\cdot\vec{g_j}-\bar{g_i}\bar{g_j} )-g_jg_i +\vec{g_i}\bar{g_j}+\vec{g_j}\bar{g_i}.
\end{eqnarray*}  
Hence the span of $B'$ contains $B$, which (by Lemma                  \ref{lem21}) generates $A_G$.
}\hfill${}$\end{proof}

We  have sets:
\begin{eqnarray*}
\widehat{B} \,&=& \{1\} \cup \{\vec{g_i} \vert\, i\in I\} \cup \{[\vec{g_i},\vec{g_j}]\vert\,i,j \in I\} \subset A_{F_I},\\
\widehat{B'}&=& \{1\} \cup \{{g_i} \vert\, i\in I\} \cup \widehat{C} \subset A_{F_I},
\end{eqnarray*}
where $\widehat{C}\subset\{g_ig_j \vert\, i,j \in I\}\subset A_{F_I}$ satisfies that for all $i,j \in I$, either $g_ig_j \in \widehat{C}$ or $g_jg_i \in \widehat{C}$.  (Note that here we are no longer identifying the $g_i$ with their images in $G$).

Taking Lemma \ref{lem20} with $J=\phi$ gives a ring homomorphism  $\widehat{\phi} \colon R_I \to \KF$:
\begin{eqnarray*}\lambda_i\,\,\,\,\,\,&\mapsto& \,\,\,\,\,\bar{g_i},\qquad\qquad\qquad\quad\,\,  i \in I,\\
 \hat\phi\colon\qquad m_{ij} \,\,&\mapsto& \,\,\,\,\,\vec{g_i} \cdot \vec{g_j}, \qquad\qquad\,\,\,\, i,j \in I,\\
w_{ijk}  &\mapsto&\,\,\,[\vec{g_i},\vec{g_j}]\cdot \vec{g_k},\, \qquad i,j,k \in I.
\end{eqnarray*}

\noindent Taking the $J=\emptyset$ case of Lemma \ref{lem21} and Lemma \ref{lem22} we get:

\begin{lemma} $\widehat{\phi}$ is surjective and $\widehat{B},\widehat{B'}$ generate $A_{F_I}$ as a module over $\KF$. 
\label{lem23}\end{lemma}

The remainder of this section will be devoted to showing that $\widehat{\phi}$ is injective and hence a ring isomorphism.  As $R_I$ was defined in terms of generators and relations, we will have a complete description for $\KF$.  Lemma \ref{lem23} gives generating sets   $\widehat{B}, \widehat{B'}$ for $A_{F_I}$ as a $\KF$--module.  Definition \ref{gen} provides generating sets for the ideal $L^\# \lhd \KF$.  Thus we will then have a complete description for $\KG \cong \KF/L^\#$.

In particular if $I$ is a finite set then $R_I$ is finitely generated and our presentation for it has a finite set of generators and relations.    Thus once we have proved the injectivity of $\widehat{\phi}$, we will have a finite description of $\KF$.  Further if $I$ is finite then so are $\widehat{B}$, $\widehat{B'}$.  Hence if $J$ is also finite, we will have a finite generating set for the ideal $L^\# \lhd \KF$ and thus a finite description of $\KG$.

Our strategy in proving that $\widehat{\phi} \colon R_I \to \KF$ is injective will be to construct another ring homomorphism $\psi\colon \KF  \to S_I$, where $S_I$ is an algebraic extension of a ring of invariants of ${\rm SO}_3(\kappa)$.  As descriptions of such rings are provided by the fundamental theorems of classical invariant theory, we may then verify that the composition $\psi\widehat{\phi}\colon R_I \to S_I$ is injective.  Thus we will have shown that $\widehat{\phi}$ is injective.

 Let $T_I$ be the polynomial  ring $\kappa[\mu_i, x_i,y_i,z_i \vert\, i \in I]$, quotiented by the relations:
$$\mu_i^2+x_i^2+y_i^2+z_i^2=1,\qquad \forall i \in I.$$
Let $H_I=T_I[e_1,e_2,e_3]$, where $e_1,e_2,e_3$ are non-commuting variables satisfying: 
$${e_1}^2={e_2}^2={e_3}^2=e_1e_2e_3=-1.$$
Note in particular that $e_1e_2=-e_2e_1=e_3$.  Also $H_I$ has a basis $\{1,e_1,e_2,e_3\}$ over $T_I$.

We define a ring homomorphism $\widehat{\psi}\colon\kF\to H_I$ by:

\begin{displaymath} 
\widehat{\psi}\colon \begin{array}{ccccc}
g_i&\mapsto& \mu_i+x_ie_1+y_ie_2+z_ie_3,&\qquad\qquad&  i \in I,\\
\quad g_i^{-1}  &\mapsto& \mu_i-x_ie_1-y_ie_2-z_ie_3, &\qquad\qquad& i \in I.
\end{array}
\end{displaymath}

\begin{lemma}{$\widehat{\psi}\colon\kF\to H_I$ is a well defined ring homomorphism.}\label{lem24}\end{lemma}

\begin{proof}Note that:  \begin{eqnarray*} \widehat{\psi}(g_i\,\,\,\,) \widehat{\psi}(g_i^{-1})&=&( \mu_i+x_ie_1+y_ie_2+z_ie_3)(\mu_i-x_ie_1-y_ie_2-z_ie_3)=1,\\
\widehat{\psi}(g_i^{-1}) \widehat{\psi}(g_i\,\,\,\,)&=&( \mu_i-x_ie_1-y_ie_2-z_ie_3)(\mu_i+x_ie_1+y_ie_2+z_ie_3)=1.\end{eqnarray*}

\end{proof}

\begin{definition}For any $q=u+ae_1+be_2+ce_3 \in H_I$, with $u,a,b,c \in T_I$, we define its \emph{conjugate} to be $q^*=u-ae_1-be_2-ce_3$.\end{definition}

One may directly verify that for  $q_1,\,q_2\in H_I$, one has $(q_1q_2)^*=q_2^*q_1^*$.

\begin{lemma}{For $\alpha \in\kF$ we have $\widehat{\psi}(\alpha^*)=(\widehat{\psi}(\alpha))^*$.}
\label{lem25}\end{lemma}

\begin{proof}{This is true when $\alpha=g_i$ or $\alpha=g_i^{-1}$ for any $i\in I$.  Clearly if it is true for $\alpha,\beta\in \kF$ then it is true for any $\kappa$--linear combination of $\alpha$ and $\beta$.  It remains to show that it is true for $\alpha\beta$:
\begin{eqnarray*}(\widehat{\psi}(\alpha\beta))^*&=&(\widehat{\psi}(\alpha)\widehat{\psi}(\beta))^*=(\widehat{\psi}(\beta))^*(\widehat{\psi}(\alpha))^*\\&=&\widehat{\psi}(\beta^*)\widehat{\psi}(\alpha^*)=
\widehat{\psi}(\beta^*\alpha^*)=\widehat{\psi}((\alpha\beta)^*).\end{eqnarray*}
}
\end{proof}

If $\alpha \in \kF^*$, then $$\widehat{\psi}(\alpha)=\widehat{\psi}(\alpha^*)=(\widehat{\psi}(\alpha))^*,$$ so
$\widehat{\psi}(\alpha)\in T_I$ and is central in $H_I$.   Thus $\widehat{\psi}$ induces a well defined ring homomorphism:
$$\widetilde{\psi}\colon A_{F_I}\to H_I.$$  

Let $S_I \subset T_I$ denote the image of $\KF \subset A_{F_I}$ under this induced map.  Then let $$\psi\colon\KF \to S_I$$ denote the restriction of $\widetilde{\psi}$ to $\KF$.

For $\alpha,\alpha' \in A_{F_I}$, we may pick $u,u',a,a',b,b',c,c'\in T_I$ such that 
\begin{eqnarray*} \widetilde{\psi}(\alpha)\,\,&=&u+ae_1+be_2+ce_3,\\\widetilde{\psi}(\alpha')&=&u'+a'e_1+b'e_2+c'e_3.\end{eqnarray*}
Then \begin{eqnarray*}{\psi}(\bar{\alpha})&=&
\frac12\biggl(\widetilde{\psi}(\alpha)+\widetilde{\psi}(\alpha^*)\biggr)=\frac12\biggl(\widetilde{\psi}(\alpha)+(\widetilde{\psi}(\alpha))^*\biggr)= u,\\
\widetilde{\psi}(\vec{\alpha})&=&
\frac12\biggl(\widetilde{\psi}(\alpha)-\widetilde{\psi}(\alpha^*)\biggr)=\frac12\biggl(\widetilde{\psi}(\alpha)-(\widetilde{\psi}(\alpha))^*\biggr)= ae_1+be_2+ce_3.
\end{eqnarray*}
Similarly, ${\psi}(\bar{\alpha}')=u'$ and $\widetilde{\psi}(\vec{\alpha}')=a'e_1+b'e_2+c'e_3$.

\begin{lemma}{Continuing with this notation we have: 

\noindent i)\,\, ${\psi}(\,\vec\alpha \cdot \vec{\alpha'}) \,= \,\,aa'+bb'+cc'$.

\noindent       ii) ${\widetilde\psi}([\vec\alpha , \vec{\alpha'}]) = (bc'-cb')e_1-(ac'-ca')e_2+(ab'-ba')e_3$.}
\label{lem26}\end{lemma}

\begin{proof}From Definition \ref{dot} we have:$$ 
\begin{array}{ccccc}{\psi}(\,\vec{\alpha} \cdot \vec{\alpha'}\,) &=&-\frac12\biggl(\widetilde\psi(\vec{\alpha})\widetilde\psi(\vec{\alpha'})+
\widetilde\psi(\vec{\alpha})\widetilde\psi(\vec{\alpha'})\biggr) \,\,\,&=& aa'+bb'+cc',\\
\widetilde{\psi}([\vec{\alpha} , \vec{\alpha'}]) &=&\,\, \frac12\biggl(\widetilde\psi(\vec{\alpha})\widetilde\psi(\vec{\alpha'})-\widetilde\psi(\vec{\alpha})\widetilde\psi(\vec{\alpha'})\biggr)
&=& \\&&(bc'-cb')e_1\quad-\quad(ac'-ca')e_2&+&(ab'-ba')e_3.\end{array}
$$

\end{proof}

We now need to show that the composition $\psi\widehat{\phi}\colon R_I \to S_I$ is injective. We know that $\widehat{\phi}$ is surjective.  By construction $\psi$ is also surjective.  Therefore $S_I$ is generated by the $\psi\widehat{\phi}(\lambda_i),\,\,\psi\widehat{\phi}(m_{ij}),\,\, \psi\widehat{\phi}(w_{ijk})$ over all $i,j,k \in I$.  By Lemma                  \ref{lem26} we have:

$$
\begin{array}{ccccccccc}
\psi\widehat{\phi}(\lambda_i)\,\,\,\,&=&{\psi}(\bar{g_i})\qquad\quad\,\,\,\,&=&\mu_i,&&i&\in& I,\\
\psi\widehat{\phi}(m_{ij})\,&=&{\psi}(\vec{g_i}\cdot \vec{g_j})\qquad&=&(x_ix_j+y_iy_j+z_iz_j),&& i,j&\in& I,\\
\psi\widehat{\phi}(w_{ijk})&=&{\psi}([\vec{g_i},\vec{g_j}]\cdot \vec{g_k})&=&\left| 
\begin{array}{ccc} 
x_i&y_i&z_i\\
x_j&y_j&z_j\\
x_k&y_k&z_k\\
\end{array} \right|, && i,j,k &\in& I.\\
\end{array}
$$

Let $V_I\subset S_I$ be the subring generated over $\kappa$ by the $\psi\widehat{\phi}(m_{ij})$ and $\psi\widehat{\phi}(w_{ijk})$.
The second fundamental theorem of the invariant theory of ${\rm SO}_3(\kappa)$ \cite [Chap. II \S17]{Weyl} states that the relations between these elements are precisely those implied by R1, R3, R4.  Note that the proof in \cite[Chap. II \S17]{Weyl} is for the case $\kappa                 =                 \R$.  However the result is clearly independent of the field of characteristic zero, $\kappa$.

From the construction of $T_I$, we have that $S_I$ is a multiple quadratic extension of $V_I$ by the elements $\{\mu_i \vert \, i\in I\}$, where each $\mu_i$ satisfies the relation implied by R2: $\mu_i^2+\psi\widehat{\phi}(m_{ii})=1$.

\begin{lemma}{The relations {\rm R1, R2, R3, R4} imply all relations between the elements $$\psi\widehat{\phi}(\lambda_i),\,\qquad\psi\widehat{\phi}(m_{ij}),\,\qquad \psi\widehat{\phi}(w_{ijk}).$$} \label{lem27} \end{lemma}

\begin{proof}{
Suppose some polynomial expression in the $\psi\widehat{\phi}(\lambda_i),\,\,\psi\widehat{\phi}(m_{ij}),\,\, \psi\widehat{\phi}(w_{ijk})$ is zero in $S_I$.  Under the relation implied by R2, this expression may be written as a $V_I$--linear combination of products of distinct $\psi\widehat{\phi}(\lambda_i)$.  As these are linearly independent over $V_I$, each coefficient must be zero as an element of $V_I$ and hence trivial under the relations implied by R1, R3, R4.}  \hfill${}$\end{proof}

\begin{lemma}{The composition  $\psi\widehat{\phi}\colon R_I \to S_I$ is an isomorphism.}
\label{lem28}\end{lemma}

\begin{proof}{The ring $S_I$ is generated by the $\psi\widehat{\phi}(\lambda_i),\,\,\psi\widehat{\phi}(m_{ij}),\,\, \psi\widehat{\phi}(w_{ijk}), i,j,k \in I$, subject to precisely the relations implied by R1, R2, R3, R4.
}\hfill${}$\end{proof}

\begin{theorem}{We have an isomorphism of rings $\phi\colon\KF \to R_I$.}\label{thm:iso}\end{theorem}

\begin{proof}{We know $\widehat{\phi}$ is surjective by Lemma                  \ref{lem23} and injective by Lemma                  \ref{lem28}.}
\hfill${}$\end{proof}

From now on we regard the isomorphism $\widehat{\phi}\colon R_I \to \KF$ as simply the identity.  That is we regard the $\lambda_i,m_{ij},w_{ijk}$ as elements of $\KF$, or indeed their images in $\KG$ depending on context.  Also we write $R_I$ and $\KF$  interchangeably. We summarize the results from this section in the following theorems:

\begin{theorem} Suppose we have a presentation $\langle g_i, i\in I \vert\, r_j, j \in J \rangle$ for a group $G$. Let $L, B,B',\widehat{B},\widehat{B'},R_I$ be as defined in this section.  

\noindent
i) We may describe $\KF$ by identifying it with $R_I$.  

\noindent
ii) Both $\widehat{B}$ and $\widehat{B'}$ are generating sets for $A_{F_I}$ as a module over $\KF$. 

\noindent
iii) We may then describe $\KG$ by identifying it with $\KF/L^\#$. 

\noindent
iv) We may take either $B$ or $B'$ as generating sets for $A_G$ as a module over $\KG$.   

\end{theorem}

\begin{theorem} In particular, if $I$  and $J$ are finite:

\noindent
i) We have a finite description of $\KF$.  

\noindent
ii) The sets $\widehat{B},\widehat{B'}$ are finite and we have a finite generating set for the ideal $L^\#$.

\noindent
iii) Thus we have a finite description of $\KG$.  In particular $\KG$ is Noetherian, and finitely generated as an algebra over $\kappa$.

\noindent
iv) Also $B,B'$ are finite sets.

\end{theorem}

Now given a finitely presented group $G$ and two finite subsets $N,M \subset G$ we are in a position to apply Theorem \ref{thm:norm} (or Theorem \ref{hashhash}) to attempt to show that they do not generate the same normal subgroup.  From the theorems above we may compute the ring $\KG$ and find finite sets of generators for the ideals $N^\#,M^\#$ (or $N^{\#\#}, M^{\#\#}$).  We are then left with the task of determining if these ideals are identical.   In the next section we demonstrate this method by giving a proof of a result to do with the normal generation of free products of pairs of cyclic groups.

 \section{Free product of two cyclic groups}\label{cyc}

We may now compute $\KG$ from a presentation of $G$.  As an example we consider: $$G=C_s \star C_t=\langle g_1, g_2 \vert\, g_1^s, g_2^t \rangle,$$
where $s,t>1$.   Here the indexing set for the generators is just $I=\{1,2\}$.  

\begin{lemma}{ $R_I$ is the polynomial ring in three variables: $R_I= \kappa[\lambda_1,\lambda_2,m_{12}]$.}
\label{lem29}\end{lemma}

\begin{proof}{From R1 we have $m_{12}=m_{21}$.  Also by R1 we have that $w_{ijk}=-w_{ijk}=0$ for all $i,j,k \in \{1,2\}$, as $i,j,k$ cannot be distinct.  R3 then becomes vacuous as does R4 (noting that a matrix with repeated rows must have zero determinant).  Finally R2 allows us to express $m_{11}, m_{22}$ in terms of $\lambda_1,\lambda_2$. 
}\hfill${}$\end{proof}

If we let $L= \{g_1^s,g_2^t\}\subset F_I$, then we have that $\KG\cong R_I/L^\#$ (by Theorems \ref{thm:norm} and \ref{thm:iso}).  To compute $L^\#$ we first need generating sets for $A_{F_I}$ over $R_I$.   By Lemma \ref{lem23}  we may take $B_{g_1^s}=\{1,g_1,g_2,g_2g_1\}$ and $B_{g_2^t}=\{1,g_2,g_1,g_1g_2\}$.  Hence by Lemma \ref{lem9} the following elements generate the ideal $L^\#$: 
\begin{eqnarray}
\begin{array}{ccccc}
\barl{g_1^s}-1,&\quad\quad \barl{g_1^{s+1}}-\barl{g_1},&\qquad\qquad&\barl{g_2g_1^s}-\barl{g_2},&\quad \quad\barl{g_2g_1^{s+1}}-\barl{g_2g_1}, \\
\barl{g_2^t}-1,&\quad \quad \barl{g_2^{t+1}}-\barl{g_2},&\qquad\qquad& \barl{g_1g_2^t}-\barl{g_1},&\quad \quad\barl{g_1g_2^{t+1}}-\barl{g_1g_2}.
\end{array} \label{gensetL}  
\end{eqnarray}

Recall Lemma \ref{cor1}, taking $k=e$:  Given $g,h\in F_I$ we have: $$\overline{gh^n}=\\ \overline{gh}P_n(\barl{h})-\overline{g}P_{n-1}(\barl{h}).$$

Thus we may express our generating set (\ref{gensetL}) as polynomial expressions in $\barl{g_1},\barl{g_2},\barl{g_1g_2},\barl{g_2g_1}$.  Further, by Lemma \ref{lem10}i and Lemma \ref{lem12}i, we may express these in terms of $\lambda_1,\lambda_2, m_{12}$: 
$$\barl{g_1}=\lambda_1,\qquad \barl{g_2}=\lambda_2,\qquad \barl{g_1g_2}=\barl{g_2g_1}=\lambda_1\lambda_2-m_{12}$$

\begin{lemma}  We conclude that $\KG$ is the ring $\kappa[\lambda_1,\lambda_2,m_{12}]$ quotiented by the set of elements \textup{(\ref{gensetL})}, expressed as polynomial expressions in the $\lambda_1,\lambda_2, m_{12}$.
\label{compdesc}
\end{lemma}

We now consider a result due to Boyer \cite{Boye}.  This states that for any $w\in C_s \star C_t$, the normal closure of a proper power $w^r$, where $r>1$, is not the whole group.  This result is an instance of a more general phenomena in combinatorial group theory, where proper powers are seen to have smaller normal closures than general elements.  In a certain sense, an element of the form $g^r$ is counted as merely $\frac 1r$ of an element as far as normal generation is concerned (see for example \cite{Allc}).  This idea is central to the Potential Counterexamples \ref{pot1} and \ref{pot2}.

We will prove Boyer's result, demonstrating that a group theoretic result proved in 1988  may be reduced to elementary algebra when one considers the commutative version of the group ring.

Let $G=  C_s \star C_t$ for $s,t>1$ and take any $w\in G$ and integer $r>1$.   Let $$f\colon G \to C_s \times C_t$$ be the  abelianisation homomorphism.   If $w^r$ did normally generate $G$ then $f(w)$ would generate  $C_s \times C_t$, so we may write $f(w)=(g_1,g_2)$ where $g_1$ is a generator for $C_s$ and $g_2$ is a generator for $C_t$.  We then have: $$G=\langle g_1, g_2 \vert\, g_1^s, g_2^t \rangle.$$

Pick a word in the letters $g_1, g_2$ to represent $w$, so that the total index of the letter $g_1$ is 1, and the total index of the letter $g_2$ is also 1 (as $f(w)=(g_1,g_2)\in C_s \times C_t$ this can be achieved by adding the correct powers of $g_1^s$ and $g_2^t$ to any word representing $w$).  We abuse notation by also denoting this word $w$.

If $w^r$ did normally generate $G$, then we would have that the normal closure of $g_1^s,g_2^t, w^r\in F_I$ was the whole group (where, as before, $I=\{1,2\}$).  That is, to show that $w^r$ does not normally generate $G$, it is sufficient by Theorem  \ref{hashhash} to show that the ideals $\{g_1^s,g_2^t,w^r\}^{\#\# }, \{g_1,g_2\}^{\#\#}\lhd \KF$ are not equal. 

 From Corollary \ref{power} we know that $\{g_1^s,g_2^t,w^r\}^{\#\# }$ is contained in the ideal generated by $P_s(\barl{g_1}), P_t(\barl{g_2}), P_r(\barl{w})$.  It is therefore sufficient to show that this ideal does not contain the elements $\barl{{g_1}^2}-1,\,\barl{{g_2}^2}-1 \in \{g_1,g_2\}^{\#\#} $.

\begin{lemma}
There exists a (possibly trivial) field extension $\F$ of $\kappa$ which contains roots $\mu_1,\mu_2$ of the polynomials $P_s,\,P_t$ respectively, as well as elements $s_1,\,s_2$ satisfying ${s_1}^2+{\mu_1}^2=1$ and ${s_2}^2+{\mu_2}^2=1$.  

Further, in any such field $\F$, we have $s_1,s_2$ invertible. \label{inv}\end{lemma}

\begin{proof}
The degrees of $P_s, P_t$ are $s-1,t-1>0$ respectively (see Definition \ref{poly}), so we may obtain $\F$ by (if necessary) adjoining to $\kappa$ the roots $\mu_1,\mu_2$, as well as the roots $s_1,s_2$ of the quadratics ${s_1}^2+{\mu_1}^2=1$ and ${s_2}^2+{\mu_2}^2=1$.

The only way $s_1$ or $s_2$ could be $0$ is if $\mu_1=\pm1$ or $\mu_2=\pm1$.  In that case we would have $P_s(1)=0$ or $P_s(-1)=0$ or $P_t(1)=0$ or $P_t(-1)=0$.

However $P_n(1)=n$ and $P_n(-1)=-1^{n+1}n$ (see Defintion \ref{poly}).
\end{proof}

\bigskip
Recall that $\KF$ is identified with $R_I$ by Theorem \ref{thm:iso}, and $R_I$ is merely $\kappa[\lambda_1,\lambda_2,m_{12}]$ by Lemma \ref{lem29}.  We have a ring homomorphism $\theta\colon \KF \to \F[x]$:
\begin{displaymath}\theta\colon
\begin{array}{ccc}
\lambda_1\,\,\,&\mapsto& \mu_1\,\,\,\quad\\
\lambda_2\,\,\,&\mapsto&\mu_2\,\,\,\quad\\
m_{12}&\mapsto&s_1s_2x\\
\end{array}
\end{displaymath}
where $\theta$ restricts to the identity on $\kappa$. (Here $\F[x]$ denotes the polynomial ring).

\bigskip

\begin{lemma}
If $G$  is  normally generated by a proper power $w^r$, then $P_r(\theta(\barl{w}))$ is a unit.\label{unit}
\end{lemma}

\begin{proof}
By construction $\theta(P_s(\barl{g_1}))=0$ and  $\theta(P_t(\barl{g_2}))=0$.  By Lemma \ref{lem10}ii:$$\theta(\barl{{g_1}^2}-1)=2\theta(\barl{g_1}^2-1)=-2{s_1}^2\,\qquad{\rm  and}\qquad \theta(\barl{{g_2}^2}-1)=2\theta(\barl{g_2}^2-1)=-2{s_2}^2.$$
Here we know $-2{s_1}^2$, $-2{s_2}^2$ are invertible by Lemma \ref{inv}.  Thus the only way that the ideal generated by $P_s(\barl{g_1}), P_t(\barl{g_2}), P_r(\barl{w})$ could contain the elements $\barl{{g_1}^2}-1,\,\barl{{g_2}^2}-1$ is if $\theta(P_r(\barl{w}))=P_r(\theta(\barl{w}))$ is a unit in $\F[x]$.
\end{proof}

As $P_r$ has strictly positive degree $r-1$, in order to show that $P_r(\theta(\barl{w}))$ is not a degree 0 polynomial in $\F[x]$, it suffices to show that $\theta(\barl{w})$ is not a degree 0 polynomial in $\F[x]$. We therefore examine $\barl{w}$ more closely.

Recall that our choice of generators $g_1,g_2$ and choice of word to represent $w$ implies that in $F_I$ we have $w=g_1g_2 C$, for some $C$ in the commutator subgroup $[F_I,F_I]\lhd F_I$.  Thus $ \barl{w}-\barl{g_1g_2}                 \in                 {\rm ker}(h^\#)$, where $h\colon F_I\to \Z \times \Z$ is the abelianisation homomorphism.

\begin{lemma}{The ideal ${\rm ker}(h^\#)$ is generated by $[\vecl{g_1},\vecl{g_2}]\cdot[\vecl{g_1},\vecl{g_2}]$.}
\label{lem36}\end{lemma}

\begin{proof}
The kernel of $h$ is normally generated by $g_1^{-1}g_2^{-1}g_1g_2\in F_I$.  Thus by Theorem \ref{thm:norm}, we have ${\rm ker}(h^\#)=\{g_1^{-1}g_2^{-1}g_1g_2 \}^\#$.

By Lemma \ref{lem21} we have the following generating set for $\Lambda_{F_I}$ over $\KF$:

\noindent $\{\vecl{g_1},\,\vecl{g_2},\, [\vecl{g_1},\vecl{g_2}]\}$.  Thus by Lemma \ref{lem:com} we have $\{g_1^{-1}g_2^{-1}g_1g_2 \}^\#$ generated by:
$$
[\vecl{g_1},\vecl{g_2}]\cdot\vecl{g_1}=0,\qquad[\vecl{g_1},\vecl{g_2}]\cdot \vecl{g_2}=0,\qquad[\vecl{g_1},\vecl{g_2}]\cdot[\vecl{g_1},\vecl{g_2}]
$$
\end{proof}

\begin{lemma}
For some polynomial $q(x)\in \F[x]$ we have: $$\theta({\bar{w}})=-s_1s_2x+\mu_1\mu_2+(1-x^2)q(x).$$\label{wbar}
\end{lemma}

\begin{proof}
By Lemma \ref{lem36} for some $Q\in \KF$ we have: $\bar{w}=\barl{g_1g_2}+ [\vecl{g_1},\vecl{g_2}]\cdot[\vecl{g_1},\vecl{g_2}]Q$.
Applying Lemma \ref{lem12}i and Lemma \ref{lem:comdotcom}i to this we get:$$
\bar{w}=-(\vecl{g_1}\cdot \vecl{g_2})+\barl{g_1}\,\barl{g_2}+((1-\barl{g_1}^2)(1-\barl{g_2}^2)-(\vecl{g_1}\cdot\vecl{g_2})^2)Q.
$$
Applying $\theta$ to both sides we obtain the desired identity.
\end{proof}

Combining Lemma \ref{unit} and Lemma \ref{wbar} we have:

\begin{theorem}\label{reduction}
Let $G=C_s \star C_t$ and let $w^r$ be a proper power of some element $w \in G$, for $r,s,t>1$.  For a field $\kappa$  let $\F$ and  $\mu_1,\mu_2,s_1,s_2\in \F$ and $q(x)\in\F[x]$ be as previously defined.  We have an implication {\bf B2} $\implies$ {\bf B1} where {\bf B1}, {\bf B2} are the statements{\rm:} 

\bigskip
\noindent {\bf B1}{\rm :}  The proper power $w^r$ does not normally generate $G$.

%\bigskip
\noindent {\bf B2}{\rm :}  The polynomial $P_r(-s_1s_2x+\mu_1\mu_2+(1-x^2)q(x))\in \F[x]$ is not a unit.
\end{theorem}

However statement {\bf B2} is an immediate consequence of the elementary result:

\begin{lemma}
Let $\F$ be any field and let $\mu_1,\mu_2,s_1,s_2\in \F$ with $s_1,s_2 \neq 0$ and let $q(x)\in \F[x]$.  For $r>1$ we have that the polynomial  \begin{eqnarray}
P_r(-s_1s_2x+\mu_1\mu_2+(1-x^2)q(x))\in \F[x]\label{polydeg}\end{eqnarray}
 has degree at least $r-1$.  Thus in particular it is not a unit in $\F[x]$.
\label{com:res}
\end{lemma}

\begin{proof}
The polynomial $P_r$ has degree $r-1$ (see Definition \ref{poly}).  Thus if $q(x)=0$ then (\ref{polydeg}) has degree $r-1$ and if $q(x)\neq 0$ then  (\ref{polydeg}) has degree at least $2(r-1)$.
\end{proof}

Thus we have an alternative proof of the following theorem:

\begin{theorem}[Boyer]\label{Boy}
The free product $C_s \star C_t$ cannot be normally generated by a proper power $w^r$, where $r,s,t>1$.
\end{theorem}

Intuitively speaking, we wished to show that repeating a word $r$ times would lead to an element in $G$ which was too `big' to normally generate the whole group.  However, because of the mysterious way in which words can unexpectedly appear in the normal closure of other words, this intuition was hard to pin down.  Applying the functor $\#$ reduced the difficult to prove statement {\bf B1} to the elementary statement {\bf B2}, that a polynomial of positive degree is too `big' to be a unit.

As the number of generators needed to generate the groups we consider goes up, so too does the complexity of the commutative algebra we are left with after applying the functor $\#$.  In the next section we consider a result involving the free product of $3$ cyclic groups.  The statement in commutative algebra which this reduces too is concise and it is possible that it may be provable by elementary commutative algebra.  However, unlike  {\bf B2},   it is not immediately obvious that this is the case.

\section{The Scott-Wiegold conjecture}\label{Scott}

The Scott-Wiegold conjecture \cite[Problem 5.53]{Kour} was an open problem for forty years, before it was proven \cite{Howi} by Jim Howie.  Thus it  illustrates how even very difficult problems in group theory can be translated into (hopefully simpler) problems in commutative algebra, by considering the commutative version of the group ring. 

We take the statement of the Scott-Wiegold conjecture ({\bf SW1}) and reduce it to a statement purely about commutative rings ({\bf SW2}) analogously to Theorem \ref{reduction}.  However this time it is not immediately obvious if it is just a case of a general result in commutative algebra, analogous to Lemma \ref{com:res}.

We will leave a more  general result as a conjecture (Conjecture \ref{alg:conj}), though we expect that it is true at least in the case when we choose our field to be $\R$.  As in Theorem \ref{reduction}, the statement {\bf SW2} need only be true over one field to imply {\bf SW1}.

Fix $r,s,t>1$ and let $G=C_r \star C_s \star C_t$.  The Scott-Wiegold Conjecture states that for any $w\in G$, the group $G$ is not normally generated by $w$.  

Suppose that it was and (as in \S\ref{cyc}) consider the abelianisation homomorphism: $$f\colon G \to C_r \times C_s \times C_t.$$  If $w\in G$ did normally generate $G$ then $f(w)$ would generate  $C_r\times C_s \times C_t$, so we may write $f(w)=(g_1,g_2,g_3)$ where $g_1$ is a generator for $C_r$, $g_2$ is a generator for $C_s$ and $g_3$ is a generator for $C_t$.  We then have: $$G=\langle g_1, g_2,g_3 \vert\, g_1^r, g_2^s,g_3^t \rangle.$$

Let $I=\{1,2.3\}$ so that $F_I$ denotes the group of words in the letters $g_1, g_2,g_3$ (and their inverses).  Pick a word in $F_I$ to represent $w$, so that the total index of the letters $g_1,g_2,g_3$ is 1 in each case (we know $f(w)=(g_1,g_2,g_3)\in C_r \times C_s \times C_t$, so this can be achieved by adding the correct powers of $g_1^r,g_2^s,g_3^t$ to any word representing $w$).  We abuse notation by also denoting this word $w\in F_I$.

If $w\in G$ did normally generate $G$, then we would have that the normal closure of $g_1^r, g_2^s,g_3^t,w\in F_I$ was the whole group.  That is, to show that $w$ does not normally generate $G$, it is sufficient by Theorem  \ref{hashhash} to show that: \begin{eqnarray}\{g_1^r, g_2^s,g_3^t,w\}^{\#\#}\neq\{g_1,g_2,g_3\}^{\#\#} \label{neqideals} \end{eqnarray} as ideals in $\KF=R_I$.

\begin{lemma}
$R_I$ is the polynomial ring in 7 variables, subject to a single relation:$$
R_I=\kappa[\lambda_1,\lambda_2,\lambda_3,m_{12},m_{23},m_{31},w_{123}]/{\rm Relation}
$$ where {\rm Relation} is given by:
$$w_{123}^2=\left| 
\begin{array}{ccc} 
1-\lambda_1^2& m_{12}&m_{31}\\
m_{12}& 1-\lambda_2^2&m_{23}\\
m_{31}& m_{23}&1-\lambda_3^2
\end{array} \right|$$
\end{lemma}

\begin{proof}
Recall from \S\ref{desc} the relations R1, R2, R3, R4 which $R_I$ is subject to.  By R1 we know that $w_{123}$ is up to sign the only non-zero $w_{ijk}$ and that for $i\neq j$, $m_{ij}$ must be $m_{12}$, $m_{23}$ or $m_{31}$.  By R2  we may express the $m_{ii}$ in terms of the $\lambda_i$.

Then R4 implies Relation, taking $i,j,k$ and $l,s,t$ to be $1,2,3$ respectively.  Noting that the determinant of a matrix is alternating in the column vectors of the matrix (resp. the row vectors of a matrix), we see that assigning different values to the indices, we get no new relations from R4.

Finally note that two of the indices $i,j,k,l$  must be identical in R3.  One may check that R1 implies that the left hand side of R3 is alternating in $i,j,k,l$.  Thus whichever values are assigned to the indices $i,j,k,l$, R3 merely gives $0=0$.
\end{proof}

For $i=1,2,3$, we have that the image of $\lambda_i$ in $\KG$ is algebraic over $\kappa$ (apply Lemma \ref{cor1} to $\barl{g_i^n}$ to see that $\lambda_iP_n(\lambda_i)-P_{n-1}(\lambda_i)-1 \mapsto 0\in \KG$, for $n=r,s,{\rm or\,}t$).  As in \S\ref{cyc}, we `absorb' these `pre-algebraic' elements by extending the field $\kappa$.

\begin{lemma}
There exists a (possibly trivial) field extension $\F$ of $\kappa$ which contains roots $\mu_1,\mu_2,\mu_3$ of the polynomials $P_r,P_s,\,P_t$ respectively, as well as elements $s_1,\,s_2,s_3$ satisfying ${s_i}^2+{\mu_i}^2=1, \,\, i=1,2,3$.  
Further, in any such field the $s_i$  are invertible.
 \label{inverse}\end{lemma}

\begin{proof}
Same as for Lemma \ref{inv}: We note that for $i=1,2,3$ we have $\mu_i \neq \pm1$, as -1,1 are not roots of $P_r$, $P_s$ or $P_t$.
\end{proof}

Let $A=\F[x,y,u,v]/ \langle(1-x^2)(1-y^2)-(u^2+v^2)\rangle$.  We define a map $\theta\colon\KF \to A$:

\bigskip
\begin{displaymath}
\theta\colon
\begin{array}{ccc}
\lambda_1\,\,\,&\mapsto& \mu_1\,\,\,\quad\\
\lambda_2\,\,\,&\mapsto&\mu_2\,\,\,\quad\\
\lambda_3\,\,\,&\mapsto&\mu_3\,\,\,\quad\\
m_{12}&\mapsto&s_1s_2x\\
m_{31}&\mapsto&s_3s_1y\\
m_{23}&\mapsto&s_2s_3(u+xy)\\
w_{123}&\mapsto&s_1s_2s_3v
\end{array}
\end{displaymath}
where $\theta$ is a ring homomorphism which restricts to the identity on $\kappa$.

\begin{lemma} The map $\theta$ is well defined.\end{lemma}

\begin{proof}
It is sufficient to verify that $\theta$ applied to both sides of Relation results in a true identity in $A$.  That is, we must check that:
$$v^2=\left| 
\begin{array}{ccc} 
1& x&y\\
x& 1&u+xy\\
y& u+xy&1
\end{array} \right|$$
Note we have cleared the factors of $s_1^2,s_2^2,s_3^2$ on both sides here,  which on the right hand side come from pulling out a factor of  $s_1$, $s_2$ and $s_3$  from the $1^{\rm st}$, $2^{\rm nd}$ and $3^{\rm rd}$ rows and columns of the matrix respectively. 
Expanding the determinant we get: $(1-x^2)(1-y^2)-u^2$, which does indeed equal $v^2$, by the defining relation in $A$.
\end{proof}

\begin{lemma}
The ideal in $A$ generated by $\theta(\{g_1,g_2,g_3\}^{\#\#})$ is the entire ring. \label{whole}
\end{lemma}

\begin{proof}
Corollary \ref{dotishashhash} gives $\vecl{g_i}\cdot\vecl{g_i}\,\, \in  \{g_1,g_2,g_3\}^{\#\#}$ for $i=1,2,3$.  But $\theta(\vecl{g_i}\cdot\vecl{g_i})$ $=\theta(1-\barl{g_i}^2)=1-\mu_i^2=s_i^2$, which is invertible in $A$ by Lemma \ref{inverse}.
\end{proof}

\begin{lemma}
The ideal generated by $\theta(\{g_1^r, g_2^s,g_3^t,w\}^{\#\#})$ is generated by $\theta(\{w\}^{\#\#})$. \label{none}
\end{lemma}

\begin{proof}
By Corollary \ref{power} we have that the ideal $\{g_1^r, g_2^s,g_3^t\}^{\#\#}$ is contained in the ideal generated by $P_r(\barl{g_1}),P_s(\barl{g_2}),P_t(\barl{g_3})$.  Applying $\theta$ to these three elements gives: $P_r(\mu_1),P_s(\mu_2),P_t(\mu_3)\in A$, which by construction equal $0$.
\end{proof}

Thus combining  Lemma \ref{whole} and Lemma \ref{none} we deduce:

\begin{lemma} If $\theta(\{w\}^{\#\#})$ generates a proper ideal in $A$ then {\rm (\ref{neqideals})} holds and hence $w$ does not normally generate $G$. 
\label{wproper}
\end{lemma}

Let $w_1,w_2,w_2',w_3,w_3'$ demote the following elements of $A$:
\begin{displaymath}
\begin{array}{ccc}
w_1=\theta(\vecl{g_1}\cdot\,\vec{w}),
& w_2=s_2^{-1}\theta([\vecl{g_1},\vecl{g_2}]\cdot\vec{w}), 
&w_2'=s_1^{-1}s_2^{-1}\theta([\vecl{g_1},\vecl{g_2}]\cdot[\vecl{g_1},\vec{w}]),\\
&w_3=s_3^{-1}\theta([\vecl{g_1},\vecl{g_3}]\cdot\vec{w}),
&w_3'=s_1^{-1}s_3^{-1}\theta([\vecl{g_1},\vecl{g_3}]\cdot[\vecl{g_1},\vec{w}]).
\end{array}
\end{displaymath}

\begin{lemma}\label{whashhash:genset}
The ideal generated by $\theta(\{w\}^{\#\#})$ has generating set:
$\{w_1,w_2,w_2',w_3,w_3'\}$.
\end{lemma}

\begin{proof}
From Lemma \ref{lem21} we obtain a generating set for $\Lambda_{F_I}$.   Then applying Corollary \ref{dotishashhash} we obtain a generating set for $\{w\}^{\#\#}$.  Applying $\theta$ yields a generating set for the ideal generated by $\theta(\{w\}^{\#\#})$:

\begin{displaymath}
\begin{array}{ccc}
v_1=\theta(\vecl{g_1}\cdot\,\vec{w}),
& v_2=\theta(\vecl{g_2}\cdot\,\vec{w}), 
&v_3=\theta(\vecl{g_3}\cdot\,\vec{w}),\\
v_{12}=\theta([\vecl{g_1},\vecl{g_2}]\cdot\vec{w}),
&v_{23}=\theta([\vecl{g_2},\vecl{g_3}]\cdot\vec{w}),
&v_{31}=\theta([\vecl{g_3},\vecl{g_1}]\cdot\vec{w})).
\end{array}
\end{displaymath}
We must first verify that the $w_i,w_i'$ all lie in the ideal generated by these elements.  Clearly, $w_1=v_1,w_2=s_2^{-1}v_{12},w_3=-s_3^{-1}v_{31}$.  By Corollary \ref{cor3}ii:
\begin{eqnarray*}{}[\vecl{g_1},\vecl{g_2}] \cdot [\vecl{g_1},\vec{w}] = 
(\vecl{g_1}\cdot \vecl{g_1})(\vecl{g_2}\cdot \vec{w})-(\vecl{g_1}\cdot \vec{w})(\vecl{g_2}\cdot \vecl{g_1}),\\
{}[\vecl{g_1},\vecl{g_3}]
\cdot [\vecl{g_1},\vec{w}]  = 
(\vecl{g_1}\cdot \vecl{g_1})(\vecl{g_3}\cdot \vec{w})-(\vecl{g_1}\cdot \vec{w})(\vecl{g_3}\cdot \vecl{g_1}).\,
\end{eqnarray*}
Applying $\theta$ and dividing through by $s_1s_2$, $s_1s_3$ respectively we get: \begin{eqnarray} w_2'=s_1s_2^{-1}v_2-xv_{1}, \qquad w_3'=s_1s_3^{-1}v_3-yv_{1}. \label{winv} \end{eqnarray}

Now we must show that the $w_i,w_i'$ generate the $v_i,v_{ij}$.  Again we have, $v_1=w_1,v_{12}=s_2w_2,v_{31}=-s_3w_3$.  Rearranging (\ref{winv}) we get:
$$ v_2=s_1^{-1}s_2w_2'+s_1^{-1}s_2xw_{1}, \qquad v_3=s_1^{-1}s_3w_3'+s_1^{-1}s_3yw_{1}.$$

It remains to show that $v_{23}$ lies in the span of the $w_i,w_i'$.  By Corollary \ref{cor3}iv:
$$([\vecl{g_1},\vecl{g_2}]\cdot\vecl{g_3})\vecl{g_1}=(\vecl{g_1}\cdot\vecl{g_1})[\vecl{g_2},\vecl{g_3}] -(\vecl{g_2}\cdot\vecl{g_1})[\vecl{g_1},\vecl{g_3}]
+(\vecl{g_3}\cdot\vecl{g_1})[\vecl{g_1},\vecl{g_2}].$$
Taking the $\cdot$ product with $\vec{w}$ and applying $\theta$ we get: $$s_1s_2s_3vw_1=s_1^2v_{23}-s_1s_2s_3xw_{3}+s_1s_2s_3yw_{2}.$$
Rearranging gives $v_{23}=s_1^{-1}s_2s_3vw_1+s_1^{-1}s_2s_3xw_{3}-s_1^{-1}s_2s_3yw_{2}$, as required.
\end{proof}

Thus by Lemma \ref{wproper}, in order to show that $w$ does not normally generate $G$, it suffices to show that $1\in A$ cannot be expressed as a linear combination of the $w_i,w_i'$.  

\begin{lemma} \label{matrix:relation}
We have the following relation between $w_2,w_2',w_3,w_3'${\rm :}
\begin{eqnarray}
\left(
\begin{array}{ccccc}
-u&v&1-x^2&0\\
-v&-u&0&1-x^2\\
1-y^2&0&-u&-v\\
0&1-y^2&v&-u\\
\end{array}
\right)
\left(
\begin{array}{c}
w_2\\w_2'\\w_3\\w_3'
\end{array}
\right)
=
\left(
\begin{array}{c}
0\\0\\0\\0
\end{array}
\right). \label{matrixequation}
\end{eqnarray}
\end{lemma}

\begin{proof}
Consider the expression:$$[[\vecl{g_1},\vecl{g_2}],\vecl{g_3}]\cdot[[\vecl{g_1},\vecl{g_2}],[\vecl{g_1},\vec{w}]].$$
Expanding this as a $\cdot$ product of two brackets using Corollary \ref{cor3}ii we get:
\begin{eqnarray*}&\,\,&
[[\vecl{g_1},\vecl{g_2}],\vecl{g_3}]\cdot[[\vecl{g_1},\vecl{g_2}],[\vecl{g_1},\vec{w}]]\\&=&
([\vecl{g_1},\vecl{g_2}]\cdot[\vecl{g_1},\vecl{g_2}]) (\vecl{g_3}\cdot[\vecl{g_1},\vec{w}])-
([\vecl{g_1},\vecl{g_2}]\cdot[\vecl{g_1},\vec{w}]) (\vecl{g_3}\cdot[\vecl{g_1},\vecl{g_2}])\\
&=&-((\vecl{g_1}\cdot\vecl{g_1})(\vecl{g_2}\cdot\vecl{g_2})-(\vecl{g_1}\cdot\vecl{g_2})^2)
 ([\vecl{g_1},\vecl{g_3}]\cdot\vec{w})-
 ([\vecl{g_1},\vecl{g_2}]\cdot\vecl{g_3})
([\vecl{g_1},\vecl{g_2}]\cdot[\vecl{g_1},\vec{w}]).
\end{eqnarray*}

On the other hand we may replace $[[\vecl{g_1},\vecl{g_2}],[\vecl{g_1},\vec{w}]]$ with $([\vecl{g_1},\vecl{g_2}]\cdot\vec{w})\vecl{g_1}$ by Corollary \ref{cor3}iii.  Thus:
\begin{eqnarray*}&\,\,&
[[\vecl{g_1},\vecl{g_2}],\vecl{g_3}]\cdot[[\vecl{g_1},\vecl{g_2}],[\vecl{g_1},\vec{w}]]\\&=&
([[\vecl{g_1},\vecl{g_2}],\vecl{g_3}]\cdot\vecl{g_1})
([\vecl{g_1},\vecl{g_2}]\cdot\vec{w})\\\qquad&=&
([\vecl{g_3},\vecl{g_1}]\cdot[\vecl{g_1},\vecl{g_2}])
([\vecl{g_1},\vecl{g_2}]\cdot\vec{w})\qquad\quad{\rm (cycling\,\,} 1^{\rm st}{\rm \,\,scalar\,\, triple\,\, product)}\\&=&
((\vecl{g_1}\cdot\vecl{g_3})(\vecl{g_1}\cdot\vecl{g_2})-(\vecl{g_1}\cdot\vecl{g_1})(\vecl{g_2}\cdot\vecl{g_3}))
([\vecl{g_1},\vecl{g_2}]\cdot\vec{w}).
\end{eqnarray*}

Equating these two expansions of $[[\vecl{g_1},\vecl{g_2}],\vecl{g_3}]\cdot[[\vecl{g_1},\vecl{g_2}],[\vecl{g_1},\vec{w}]]$, applying $\theta$ to both sides, and dividing through by $s_1^2s_2^2s_3$, yields the top row of (\ref{matrixequation}).

If instead we equate the two expansions of $[[\vecl{g_1},\vecl{g_2}],\vecl{g_3}]\cdot[[\vecl{g_1},\vecl{g_2}],[\vecl{g_1},\vec{w}]]$, but interchange all $\vecl{g_2}$'s and $\vecl{g_3}$'s, before applying $\theta$ to both sides, and dividing through by $s_1^2s_2s_3^2$, we obtain the third row of (\ref{matrixequation}).

Now consider the expression:$$
[[\vecl{g_1},\vecl{g_2}],[\vecl{g_1},\vecl{g_3}]]\cdot [[\vecl{g_1},\vecl{g_2}],[\vecl{g_1},\vec{w}]].
$$
By Corollary \ref{cor3}iii we substitute $([\vecl{g_1},\vecl{g_2}]\cdot\vecl{g_3})\vecl{g_1}$ for $[[\vecl{g_1},\vecl{g_2}],[\vecl{g_1},\vecl{g_3}]]$ and $([\vecl{g_1},\vecl{g_2}]\cdot\vec{w})\vecl{g_1}$ for  $[[\vecl{g_1},\vecl{g_2}],[\vecl{g_1},\vec{w}]]$:
\begin{eqnarray*}
&\,\,&[[\vecl{g_1},\vecl{g_2}],[\vecl{g_1},\vecl{g_3}]]\cdot [[\vecl{g_1},\vecl{g_2}],[\vecl{g_1},\vec{w}]]\\&=&
(\vecl{g_1}\cdot\vecl{g_1})([\vecl{g_1},\vecl{g_2}]\cdot\vecl{g_3})([\vecl{g_1},\vecl{g_2}]\cdot\vec{w}).
\end{eqnarray*}

On the other hand we could apply Corollary \ref{cor3}ii to expand the expression as a $\cdot$ product of two brackets:
\begin{eqnarray*}
&\,\,&[[\vecl{g_1},\vecl{g_2}],[\vecl{g_1},\vecl{g_3}]]\cdot [[\vecl{g_1},\vecl{g_2}],[\vecl{g_1},\vec{w}]]\\&=&
([\vecl{g_1},\vecl{g_2}]\cdot[\vecl{g_1},\vecl{g_2}])([\vecl{g_1},\vecl{g_3}]\cdot[\vecl{g_1},\vec{w}])
-([\vecl{g_1},\vecl{g_2}]\cdot[\vecl{g_1},\vecl{g_3}])([\vecl{g_1},\vecl{g_2}]\cdot[\vecl{g_1},\vec{w}])\\&=&
((\vecl{g_1}\cdot\vecl{g_1})(\vecl{g_2}\cdot\vecl{g_2})-(\vecl{g_1}\cdot\vecl{g_2})^2)([\vecl{g_1},\vecl{g_3}]\cdot[\vecl{g_1},\vec{w}])\\&&
-((\vecl{g_1}\cdot\vecl{g_1})(\vecl{g_2}\cdot\vecl{g_3})
-(\vecl{g_1}\cdot\vecl{g_3})(\vecl{g_1}\cdot\vecl{g_2}))
([\vecl{g_1},\vecl{g_2}]\cdot[\vecl{g_1},\vec{w}]).
\end{eqnarray*}

Equating these two expansions of $[[\vecl{g_1},\vecl{g_2}],[\vecl{g_1},\vecl{g_3}]]\cdot [[\vecl{g_1},\vecl{g_2}],[\vecl{g_1},\vec{w}]]$, applying $\theta$ to both sides, and dividing through by $s_1^3s_2^2s_3$, yields the second row of (\ref{matrixequation}).

If instead we equate the two expansions of $[[\vecl{g_1},\vecl{g_2}],[\vecl{g_1},\vecl{g_3}]]\cdot [[\vecl{g_1},\vecl{g_2}],[\vecl{g_1},\vec{w}]]$, but interchange all $\vecl{g_2}$'s and $\vecl{g_3}$'s, before applying $\theta$ to both sides, and dividing through by $s_1^3s_2s_3^2$, we obtain the fourth row of (\ref{matrixequation}). 
\end{proof}

We will give a generating set for the kernel of the matrix in Lemma \ref{matrix:relation}.  First we need to better understand the ring $A$.

\begin{lemma}\label{quad}
We may regard $A$ as a quadratic extension of the polynomial ring $\F[x,y,u]${\rm ;} so every element of $A$ may written uniquely in the form $a+a'v$ with $a,a'\in\F[x,y,u]$.
\end{lemma}

\begin{proof}
Given any element of $A$ we may regard it as a polynomial in $v$, with coefficients in $\F[x,y,u]$.  If an element of $A$ has degree greater than $1$ as a polynomial in $v$, then the defining relation in $A$: $(1-x^2)(1-y^2)=u^2+v^2$, allows us to represent the element  as a polynomial in $v$ of lesser degree.  Any element of $A$ may be represented as a polynomial in $v$ of minimal degree, which will then have the required form.

To see that this is unique, suppose we have two degree at most $1$ polynomials in $v$ representing the same element of $A$.  Then the difference must be at most degree $1$ as a polynomial in $v$.  However it must also be divisible  by the degree $2$  polynomial $(1-x^2)(1-y^2)-u^2-v^2$.  Thus the difference is $0$.
\end{proof}

Let $M$ denote the matrix:\begin{eqnarray}M=\left(
\begin{array}{ccccc}
u&v&1-x^2&0\\
-v&u&0&1-x^2\\
1-y^2&0&u&-v\\
0&1-y^2&v&u\\
\end{array}
\right)\label{M}
\end{eqnarray}

\bigskip
\begin{lemma}
The columns of the matrix $M$ generate the kernel of the matrix in {\rm (\ref{matrixequation})}.  \label{kernelisimage}
\end{lemma}

\begin{proof}
Direct calculation shows that the columns of $M$ lie in the kernel.  

Now suppose that:\begin{eqnarray} \left(
\begin{array}{ccccc}
-u&v&1-x^2&0\\
-v&-u&0&1-x^2\\
1-y^2&0&-u&-v\\
0&1-y^2&v&-u\\
\end{array}
\right)
\left(
\begin{array}{c}
z_1\\z_2\\z_3\\z_4\end{array}
\right)=
\left(
\begin{array}{c}
0\\0\\0\\0\end{array}
\right) \label{matrixeq} 
\end{eqnarray}

Reducing the third row of (\ref{matrixeq}) modulo the ideal $I=\langle u,v,1-x^2\rangle$ we get $(1-y^2)z_1=0$.  However $1-y^2$ is not a zero divisor in the ring $A/I=\F[x,y]/\langle1-x^2\rangle$.  Thus $z_1 \in I$, and subtracting some $A$-linear combination of the first three columns of $M$ from $(z_1,z_2,z_3,z_4)^T$, we may assume that $z_1=0$.

Then the second row of (\ref{matrixeq}) gives $uz_2=(1-x^2)z_4$.  Representing $z_2,z_4$ as in Lemma \ref{quad} and noting that the polynomial ring $\F[x,y,u]$ is a unique factorisation domain, we have that $1-x^2$ divides $z_2$.  Thus subtracting some multiple of the fourth column of $M$ from $(z_1,z_2,z_3,z_4)^T$, we may assume that $z_1=z_2=0$.

Then the first and second row of (\ref{matrixeq}) give $(1-x^2)z_3=0$ and $(1-x^2)z_4=0$.  Representing $z_3,z_4$ as in Lemma \ref{quad} and noting that $1-x^2$ is not a zero divisor in the polynomial ring $\F[x,y,u]$, we conclude that $z_3=z_4=0$.

Thus starting with an arbitrary $(z_1,z_2,z_3,z_4)^T$ satisfying (\ref{matrixeq}) and subtracting a linear combination of the columns of $M$, we are left with $(0,0,0,0)^T$.
\end{proof}

\begin{definition}
We say that an element of $A$ is \emph{represented} by $M$ if it lies in the image of the bilinear map represented by $M$.  That is $a \in A$ is represented by $M$ precisely when there exist ${\bf q_1},{\bf q_2}\in A^4$ such that  $a={\bf q_1}^T M {\bf q_2}$.
\end{definition}

\begin{lemma}\label{w2w3constraint}
Any linear combination of $w_2,w_2',w_3,w_3'$ is represented by $M$.
\end{lemma}

\begin{proof}
By Lemma \ref{matrix:relation} we have that $(w_2,w_2',w_3,w_3')^T$ satisfies (\ref{matrixeq}).  Thus by Lemma \ref{kernelisimage} we have $(w_2,w_2',w_3,w_3')^T=M{\bf q_2}$ for some ${\bf q_2}\in A^4$.  So if $a \in A$ is a linear combination of $w_2,w_2',w_3,w_3'$, then $a={\bf q_1}^T M {\bf q_2}$ for some ${\bf q_1}\in A^4$.
\end{proof}

Having thus constrained linear combinations of $w_2,w_2',w_3,w_3'$, we turn our attention to the element $w_1 =\theta(\vecl{g_1}\cdot\,\vec{w})\in A$.  Our approach is similar to that in \S\ref{cyc}, in that we first consider the abelianisation homomorphism:$$h\colon F_I\to \Z\times\Z\times  \Z.$$

\begin{lemma}\label{lem:diff}
We have $(\vecl{g_1}\cdot\,\vec{w})-(\vecl{g_1}\cdot\,\vecl{g_1g_2g_3})\in{\rm ker}(h^\#)$.
\end{lemma}

\begin{proof}
By Lemma \ref{lem12}i we have:\begin{eqnarray}(\vecl{g_1}\cdot\,\vec{w})-(\vecl{g_1}\cdot\,\vecl{g_1g_2g_3})=
-\frac12((\barl{g_1w}-\barl{g_1^2g_2g_3})- (\barl{g_1^{-1}w}-\barl{g_2g_3})).
\label{inkerh}\end{eqnarray}
Recall that the word $w\in F_I$ was constructed so that $h(w)=h(g_1g_2g_3)$.  Thus $h(g_1w)=h(g_1^2g_2g_3)$ and $h(g_1^{-1}w)=h(g_2g_3)$.  We conclude:$$h^{\#}(\barl{g_1w}-\barl{g_1^2g_2g_3})=0,\qquad h^{\#}(\barl{g_1^{-1}w}-\barl{g_2g_3})=0,$$
so by (\ref{inkerh}) we get $h^\#((\vecl{g_1}\cdot\,\vec{w})-(\vecl{g_1}\cdot\,\vecl{g_1g_2g_3}))=0$.
\end{proof}

Hence we know that $w_1$ has the form $\theta(\vecl{g_1}\cdot\,\vecl{g_1g_2g_3})+\alpha$ for some $\alpha$ in the ideal $J$ generated by $\theta({\rm ker} (h^\#))$.  We will proceed to give a generating set for the ideal $J$, but first we expand $\theta(\vecl{g_1}\cdot\,\vecl{g_1g_2g_3})$.

Fix $W\in A$ to be the following:
\begin{eqnarray}
W= -s_1^2s_2s_3xy+\mu_1\mu_3s_1s_2x+\mu_1\mu_2s_1s_3y+s_1^2\mu_2\mu_3.\label{W} \end{eqnarray}

\begin{lemma}\label{lem:exp}
We have $\theta(\vecl{g_1}\cdot\,\vecl{g_1g_2g_3})=W-s_1^2s_2s_3u+\mu_1s_1s_2s_3v$.
\end{lemma}

\begin{proof}
By Lemma \ref{lem15} we have $g_2g_3=\barl{g_2}\,\barl{g_3}-(\vecl{g_2}\cdot\vecl{g_3})+\barl{g_3}\vecl{g_2}+\barl{g_2}\vecl{g_3}+[\vecl{g_2},\vecl{g_3}]$.  Again applying Lemma \ref{lem15} we have:
$$\vecl{g_1g_2g_3}=(\barl{g_2}\,
\barl{g_3}-\vecl{g_2}\cdot\vecl{g_3})\vecl{g_1}+\barl{g_1}(\barl{g_3}\vecl{g_2}+\barl{g_2}\vecl{g_3}+[\vecl{g_2},\vecl{g_3}])+[\vecl{g_1},\barl{g_3}\vecl{g_2}+\barl{g_2}\vecl{g_3}+[\vecl{g_2},\vecl{g_3}]].
$$
Taking the $\cdot$ product with $\vecl{g_1}$, the last term vanishes.  Then applying $\theta$ to both sides yields the result.
\end{proof}

\begin{lemma}
The ideal $J$ generated by $\theta({\rm ker}(h^\#))$ is generated by: $u,v,1-x^2,1-y^2$.
\end{lemma}

\begin{proof}
As $F_I$ has three generators $g_1,g_2,g_3$, the kernel of $h$ is normally generated by the commutators $g_1^{-1}g_2^{-1}g_1g_2,\, g_1^{-1}g_3^{-1}g_1g_3,\, g_2^{-1}g_3^{-1}g_2g_3$.  Thus by Theorem \ref{thm:norm}: $${\rm ker}(h^\#)=\{g_1^{-1}g_2^{-1}g_1g_2,\, g_1^{-1}g_3^{-1}g_1g_3,\, g_2^{-1}g_3^{-1}g_2g_3\}^\#.$$   Taking the generating set for $\Lambda_{F_I}$ from Lemma \ref{lem21} and applying Lemma \ref{lem:com} we obtain a generating set for ${\rm ker}(h^\#)$ given by taking the $\cdot$ products of $$[\vecl{g_1},\vecl{g_2}],[\vecl{g_1},\vecl{g_3}],[\vecl{g_2},\vecl{g_3}]\qquad {\rm with}\qquad \vecl{g_1},\vecl{g_2},\vecl{g_3},[\vecl{g_1},\vecl{g_2}],[\vecl{g_1},\vecl{g_3}],[\vecl{g_2},\vecl{g_3}].$$

Avoiding zero's and repeats this leaves a generating set of seven elements for ${\rm ker}(h^\#)$.   After applying $\theta$ to these seven elements (using Corollary \ref{cor3}ii to simplify) and rescaling by units in $\F$, we are left with a generating set for $J$:$$u,v,1-x^2,1-y^2, \qquad 1-(u+xy)^2,\qquad x(u+xy)-y,\qquad x-y(u+xy).$$

To complete the proof we need only note that:
\begin{eqnarray*}
1-(u+xy)^2&=&(1-x^2)+x^2(1-y^2)-(u+2xy)u,\\
x(u+xy)-y&=& xu-y(1-x^2),\\
x-y(u+xy)&=&x(1-y^2)-yu.
\end{eqnarray*}
\end{proof}

\begin{lemma}\label{w1constraint}
We have $w_1=W+\alpha$, for some $\alpha\in J$.
\end{lemma}

\begin{proof}
By Lemma \ref{lem:diff} we have $w_1-\theta(\vecl{g_1}\cdot\,\vecl{g_1g_2g_3})\in J$.  Thus by Lemma \ref{lem:exp}: $$ w_1-W+s_1^2s_2s_3u-\mu_1s_1s_2s_3v\in J.$$
Noting that $s_1^2s_2s_3u-\mu_1s_1s_2s_3v\in J$ we have $w_1-W \in J$.
\end{proof}

\begin{theorem}\label{SW1SW2}
For $r,s,t>1$ let $G=C _r\star C_s \star C_t$, let the field $\F$ and elements $\mu_1,\mu_2,\mu_3,s_1,s_2,s_3\in \F$ be as in  Lemma \ref{inverse} and let $A=\F[x,y,u,v] / \langle(1-x^2)(1-y^2)-(u^2+v^2)\rangle$.  Further let $J \lhd A$ be the ideal generated by $\{u,v,1-x^2,1-y^2\}$, let $W\in A$ be as given by (\ref{W}) and let the matrix $M$ be as given by (\ref{M}). 

We have an implication of statements {\bf SW2} $\implies$ {\bf SW1} where {\bf SW1}, {\bf SW2} are{\rm :}

\bigskip
\noindent {\bf SW1}{\rm :} No $w \in G$ normally generates $G$.

\noindent {\bf SW2}{\rm :} Given $a\in A$ represented by $M$ and $\alpha \in J$, the elements $\{a,W+\alpha\}$ generate a proper ideal in $A$. 
\end{theorem}

\begin{proof}
Assume that {\rm SW2} is true.  Suppose $w \in G$ normally generates $G$.  Let $a$ be any linear combination of $w_2,w_2',w_3,w_3'$.  Lemma \ref{w2w3constraint} tells us $a$ is represented by $M$.    Lemma \ref{w1constraint} tells us $w_1=W+\alpha$ for some $\alpha \in J$.  Then {\bf SW2} implies that we cannot write $1=\lambda w_1 +a$ with  $\lambda \in A$.  Thus $w_1,w_2,w_2',w_3,w_3'$ generate a proper ideal in $A$.  

From Lemma \ref{whashhash:genset} we know that this proper ideal is precisely the ideal generated by $\theta(\{w\}^{\#\#})$ .  From Lemma \ref{wproper} we conclude that $w$ does not normally generate $G$.
\end{proof}

Thus applying the functor $\#$ to the groups involved in the Scott-Wiegold conjecture reduces the problem to proving {\bf SW2}, a statement in commutative algebra.  Unlike {\bf B2} this statement is not immediately obvious, so we cannot claim this time, that the theory associated to $\#$ immediately proves the result.

Note however that Theorem \ref{SW1SW2} holds for any field $\F$ satisfying the properties demanded by Lemma \ref{inverse}.  That is {\bf SW2} only needs to hold for one such field in order to deduce the Scott-Wiegold conjecture.  We expect that modifying Howie's original proof \cite{Howi}, one may take $\F=\R$ and have $\vert s_1^2s_2s_3\vert> \vert s_1s_2\mu_1\mu_3\vert+\vert s_1s_3\mu_1\mu_2\vert+\vert s_1^2\mu_2\mu_3\vert$, and prove {\bf SW2} in this case, thus proving the Scott-Wiegold conjecture.

This may seem a long-winded approach given the concise nature of the proof in \cite{Howi}.  One benefit of doing it this way is that the topological arguments in \cite{Howi}, when applied to statements about ideal generation such as {\bf SW2}, can be understood algebraically in terms of Euler classes / Euler class groups \cite{Bhat, Kong}, and weak Mennicke symbols \cite{Kall}.  

However what would really be illuminating is a proof of a more general algebraic statement which implies {\bf SW2}, analogous to how Lemma \ref{com:res}  implies {\bf B2}.  Ideally:

\begin{conjecture}\label{alg:conj}  Let $\F$ be any field, let $$A=\F[x,y,u,v] / \langle(1-x^2)(1-y^2)-(u^2+v^2)\rangle,$$ and let $J \lhd A$ be the ideal generated by $\{u,v,1-x^2,1-y^2\}$.  Let $M$ be as defined in (\ref{M}) and let $W'=c_3xy+c_2y+c_1x+c_0,$ with $c_0,c_1,c_2,c_3 \in \F$ and $c_3 \neq 0$.  

Given $a\in A$ represented by $M$ and $\alpha\in J$, the elements $\{a,W'+\alpha\}$ generate a proper ideal in $A$.

\end{conjecture}

If a result such as this could be proved then we would understand, at the level of commutative group rings, the obstruction to normally generating $C_r \star C_s \star C_t$ by a single element.  Then having attained results using $\#$ about $3$-generated groups as well as $2$-generated groups (Theorem \ref{Boy}), we could hope to approach the $4$-generated groups from Potential counterexamples \ref{pot1} and \ref{pot2} by these means.  As discussed in \S1, if successful this could result in the resolution of several conjectures; in the first place the Wiegold conjecture and the Relation Gap problem.

\end{document}